\documentclass[subeqn]{amsart}
\usepackage{amsmath}
\usepackage{amssymb}
\usepackage{mathtools}
\usepackage[latin1]{inputenc}
\usepackage[swedish,english]{babel}
\usepackage{enumerate}
\usepackage{ifpdf}
\ifpdf
  \usepackage[pdftex]{graphicx}
\else
  \usepackage[dvips]{graphicx}
\fi
\usepackage[pdftitle={On the stability of stochastic jump kinetics},
  pdfauthor={Stefan Engblom},
  pdffitwindow=true,
  breaklinks=true,
  colorlinks=true,
  urlcolor=blue,
  linkcolor=red,
  citecolor=red,
  anchorcolor=red]{hyperref}
\usepackage[numbers]{natbib}

\newcommand{\Ordo}[1]{\mathcal{O} \left( #1 \right)}
\newcommand{\natop}[2]{\genfrac{}{}{0pt}{1}{#1}{#2}}
\newcommand{\lognorm}[1]{\mu_{2} \left[ #1 \right]}

\newcommand{\Realdom}{\mathbf{R}}
\newcommand{\Intdom}{\mathbf{Z}}
\newcommand{\stoich}{\mathbb{N}}
\newcommand{\hatw}{\hat{w}}
\newcommand{\hatF}{\hat{F}}
\newcommand{\Vol}{V}
\newcommand{\stopping}{P}

\newcommand{\Master}{\mathbb{M}^{T}}
\newcommand{\Masteradj}{\mathbb{M}}

\newcommand{\Probspace}{\Omega}
\newcommand{\Probelem}{\omega}
\newcommand{\Probfiltr}{\mathcal{F}}
\newcommand{\Prob}{\mathbf{P}}
\newcommand{\markspace}{I}
\newcommand{\loc}{\mathrm{loc}}
\newcommand{\Sspace}[1]{S_{\Probfiltr}^{#1,\loc}(\Intdom_{+}^{D})}

\newcommand{\onevec}{\boldsymbol{1}}
\newcommand{\onevect}{\boldsymbol{1}^{T}}
\newcommand{\lvec}{\boldsymbol{l}}
\newcommand{\lvect}{\boldsymbol{l}^{T}}
\newcommand{\lnorm}[1]{\left\| #1 \right\|_{\lvec}}

\numberwithin{equation}{section}
\numberwithin{table}{section}
\numberwithin{figure}{section}

\theoremstyle{plain}
\newtheorem{theorem}{Theorem}[section]
\newtheorem{lemma}[theorem]{Lemma}
\newtheorem{proposition}[theorem]{Proposition}
\newtheorem{corollary}[theorem]{Corollary}

\theoremstyle{definition}
\newtheorem{definition}{Definition}[section]
\newtheorem{assumption}[definition]{Assumption}
\newtheorem{example}[definition]{Example}

\begin{document}

\title[Stability of stochastic jump kinetics]
      {On the stability of stochastic jump kinetics}

\author[S. Engblom]{Stefan Engblom}
\address{Division of Scientific Computing \\
  Department of Information Technology \\
  Uppsala University \\
  SE-751 05 Uppsala, Sweden.}

\email{\href{mailto:stefane@it.uu.se}{stefane@it.uu.se}}
\thanks{Corresponding author: S. Engblom, telephone +46-18-471 27 54,
  fax +46-18-51 19 25.}


%


\date{October 17, 2014}

\selectlanguage{english}

\begin{abstract}
  Motivated by the lack of a suitable constructive framework for
  analyzing popular stochastic models of Systems Biology, we devise
  conditions for existence and uniqueness of solutions to certain jump
  stochastic differential equations (SDEs). Working from simple
  examples we find \emph{reasonable} and \emph{explicit} assumptions
  on the driving coefficients for the SDE representation to make
  sense. By `reasonable' we mean that stronger assumptions generally
  do not hold for systems of practical interest. In particular, we
  argue against the traditional use of global Lipschitz conditions and
  certain common growth restrictions. By `explicit', finally, we like
  to highlight the fact that the various constants occurring among our
  assumptions \emph{all can be determined once the model is fixed}.

  We show how basic long time estimates and some limit results for
  perturbations can be derived in this setting such that these can be
  contrasted with the corresponding estimates from deterministic
  dynamics. The main complication is that the natural path-wise
  representation is generated by a counting measure with an intensity
  that depends nonlinearly on the state.
\end{abstract}

\maketitle


\section{Introduction}

The observation that detailed modeling of biochemical processes inside
living cells is a close to hopeless task is a strong argument in favor
of stochastic models. Such models are often thought to be more
accurate than conventional rate-diffusion laws, yet remain more
manageable than, say, descriptions formed at the level of individual
molecules. Indeed, several studies \cite{fluctuation_limits,
  stochgeneexpression, Arkin} have showed that noisy models have the
ability to capture relevant phenomena and to explain actual, observed
dynamics.

In this work we shall consider some `flow' properties of a stochastic
dynamical system in the form of a quite general continuous-time Markov
chain. Since the pioneering work of Gillespie \cite{gillespie,
  gillespieCME}, in the Systems Biology context this type of model is
traditionally described in terms of a (chemical) \emph{master
  equation} (CME). This is the forward Kolmogorov equation of a
certain jump stochastic differential equation (jump SDE for brevity),
driven by independent point processes with state-dependent
intensities. Despite the popularity of the master equation approach,
little analysis on a per trajectory-basis of actual models has been
attempted.

In the general literature, when discussing existence/uniqueness and
various types of perturbation results, different choices of
assumptions with different trade-offs have been made. One finds that
the treatment often falls into one of two categories taking either a
``mathematical'' or a ``physical'' viewpoint. Either the conditions
are highly general but with subsequently less transparent proofs and
resulting in more abstract bounds. Or the conditions are formed out of
convenience, say, involving global Lipschitz constants, and classical
arguments carry through with only minor modifications.

Protter \cite[Chap.~V]{protterSDE} offers a nice discussion from the
mathematical point of view and in ascending order of generality,
including the arguably highly unrestrictive assumption of locally
Lipschitz continuous coefficients. Other authors
\cite[Chap.~6]{LevySDEs}, \cite[Chap.~3--5]{jumpSDEs} also treat the
evolution of general jump-diffusion SDEs in \emph{continuous} state
spaces.

A study of the flow properties of jump SDEs is found in
\cite{jumpSDE_wellposed}, where the setting is scalar and the state
continuous. In \cite{jumpSPDEnonLipschitz} jump stochastic
\emph{partial} differential equations are treated, and
existence/uniqueness results as well as ergodic results for the case
of a multiplicative noise, are found in
\cite{jumpSPDE_wellposed1,jumpSDEbound}. Numerical aspects in a
similar setting are discussed in \cite{jumpSPDE_wellposed2}.

In a more applied context, stability is often thought of as implied
from physical premises and the solution is tactically assumed to be
confined inside some bounded region \cite[Chap.~V]{VanKampen}. The
fundamental issue here is that for \emph{open} systems in a stochastic
setting, there is a non-zero probability of reaching \emph{any} finite
state and global assumptions must be formed with great care. The
analysis of open networks under an \textit{a priori} assumption of
boundedness is therefore quite difficult to interpret other than in a
qualitative sense. Notable examples in this setting include time
discretization strategies \cite{tau_higham, tau_li}, time-parallel
simulation techniques \cite{ssa_parareal}, and parameter perturbations
\cite{parameterCTMC}.

Evidently, essentially no systems of interest satisfy global Lipschitz
assumptions since the fundamental \emph{interaction} almost always
takes the form of a quadratic term. Interestingly, for \emph{ordinary}
differential equations, it has been shown \cite{computabilityODE} that
Lipschitz continuous coefficients imply a computationally
polynomial-space complete solution; thus providing a kind of
explanation for the convenience with this \emph{weak feedback}
assumption. It is also known \cite{explSDE}, that with SDEs,
superlinearly growing coefficients may in fact cause the forward Euler
method to diverge.

\subsection{Agenda}

Besides its expository material, the purpose of this paper is to
devise simple conditions that imply stability for finite and, in
certain cases, infinite times, and that, when applied to \emph{systems
  of practical interest}, yield explicit expressions for the
associated stability estimates. As a result the framework developed
herein applies in a constructive way to any chemical network, of
arbitrary size and topology, formed by any combination of the
elementary reactions \eqref{eq:el} to be presented in
Section~\ref{sec:stochkinetics}. Additionally, it will be clear how to
encompass also other types of nonlinear reactions that typically
result from adiabatic simplifications.

As an argument in favor of this bottom-up approach one can note that,
for evolutionary reasons, biochemical systems tend to operate close to
critical points in phase-space where the efficiency is the
highest. Clearly, for such dynamical systems, an analysis \emph{by
  analogy} might be highly misleading.

We also like to argue that our results are of interest from the
modeling point of view. Due to the type of phenomenological arguments
often involved, judging the relative effect of the (non-probabilistic)
\emph{epistemic uncertainty} is a fundamental issue which has so far
not rendered a consistent analysis.

\subsection{Outline}

The expository material in Section~\ref{sec:stochkinetics} is devoted
to formulating the type of processes we are interested in. We state
the master equation as well as the corresponding jump SDE and we also
look at some simple, yet informative actual examples. Since it is
expected that the properties of the stochastic dynamics are somehow
similar to those of the deterministic version, we search for a set of
minimal assumptions in the latter setting in
Section~\ref{sec:det}. Techniques for finding explicit values of the
constants occurring among our assumptions are also devised. The main
results of the paper are found in Section~\ref{sec:stoch} where we put
our theory together and prove existence and uniqueness, as well as
long time estimates and limit results for perturbations. A concluding
discussion is found in Section~\ref{sec:conclusions}.

%
%
%
%


\section{Stochastic jump kinetics}
\label{sec:stochkinetics}

In this section we start with the physicist's traditional viewpoint of
pure jump processes and write down the governing \emph{master
  equations}. These are evolution equations for the probability
densities of continuous-time Markov chains over a discrete
state space. Although the application considered here is mesoscopic
chemical kinetics, identical or very similar stochastic models are
also used in Epidemiology \cite{nlinepidemics}, Genetics
\cite{mathgenetics} and Sociodynamics \cite{extinction_rdme}, to name
just a few.

We then proceed with discussing a path-wise representation in terms of
a stochastic jump differential equation. The reason the sample path
representation is interesting is the possibility to reason about
\emph{flow} properties and thus compare functionals of single
trajectories. This is generally not possible with the master equation
approach.

For later use we conclude the section by looking at some prototypical
models. A simple analysis shows, somewhat surprisingly, that an
innocent-looking example produces second moments that grow
indefinitely.

\subsection{Reaction networks and the master equation}
\label{subsec:network}

We consider a chemical network consisting of $D$ different
\emph{chemical species} interacting according to $R$ prescribed
\emph{reaction pathways}. At any given time $t$, the \emph{state} of
the system is an integer vector $X(t) \in \Intdom_{+}^{D} = \{0,1,2,
\ldots \}^{D}$ counting the number of individual molecules of each
species. A reaction law is a prescribed change of state with an
intensity defined by a \emph{reaction propensity}, $w_{r}:
\Intdom_{+}^{D} \to \Realdom_{+}$. This is the transition probability
per unit of time for moving from the state $x$ to $x-\stoich_{r}$;
\begin{align}
  \label{eq:prop}
  \Prob\left[X(t+dt) = x-\stoich_{r}| \; X(t) = x\right] &=
  w_{r}(x) \, dt+o(dt).
\end{align}
where $\stoich_{r} \in \Intdom^{D}$ is the transition step and is the
$r$th column in the \emph{stoichiometric matrix} $\stoich \in
\Intdom^{D \times R}$. Informally, for states $x(t) \in
\Realdom_{+}^{D}$, we can picture \eqref{eq:prop} as a stochastic
version of the time-homogeneous ordinary differential equation
\begin{align}
  \label{eq:ODE}
  x'(t) &= -\sum_{r = 1}^{R} \stoich_{r} w_{r}(x) = 
  -\stoich w(x) =: F(x),
\end{align}
where $w(x) \equiv [w_{1}(x), \ldots, w_{R}(x)]^{T}$ is the column
vector of reaction propensities.

The physical premises leading to a description in the form of discrete
transition laws \eqref{eq:prop} often imply the existence of a
\emph{system size} $\Vol$ (e.g.~physical volume or total number of
individuals). For instance, in a given volume $\Vol$ the
\emph{elementary} chemical reactions can be written using the state
vector $x = [a,b]^{T}$,
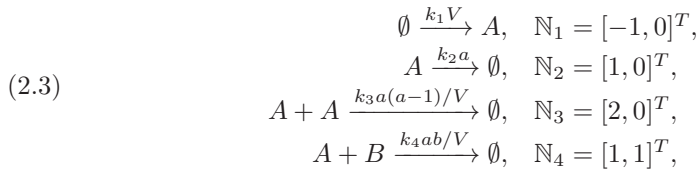
\begin{align}
  \label{eq:el}
  \begin{array}{rl}
    \emptyset \xrightarrow{k_{1} \Vol} A, & \stoich_{1} = [-1,0]^{T}, \\
    A \xrightarrow{k_{2} a} \emptyset, & \stoich_{2} = [1,0]^{T}, \\
    A+A \xrightarrow{k_{3} a(a-1)/\Vol} \emptyset, & \stoich_{3} = [2,0]^{T}, \\
    A+B \xrightarrow{k_{4} ab/\Vol} \emptyset, & \stoich_{4} = [1,1]^{T},
  \end{array}
\end{align}
with the names of the species in capitals. These propensities are
generally scaled such that $w_{r}(x) = \Vol u_{r}(x/\Vol)$ for some
dimensionless function $u_{r}$. Intensities of this form are called
\emph{density dependent} and arise naturally in a number of situations
\cite[Chap.~11]{Markovappr}. For the rest of this paper, we
conveniently take $\Vol = 1$ and defer system's size analysis to
another occasion.


The models we consider here all have states in the positive integer
lattice and the assumption that no transition can yield a state
outside $\Intdom_{+}^{D}$ is therefore natural. We make this formal as
follows \cite[Chap.~8.2.2, Definition~2.4]{BremaudMC}:
\begin{assumption}[\textit{Conservation and stability}]
  \label{ass:prop0}
  For all propensities, $w_{r}(x) = 0$ for any $x \in \Intdom_{+}^{D}$
  such that $x-\stoich_{r} \not \in \Intdom_{+}^{D}$, and we also
  restrict initial data to $\Intdom_{+}^{D}$. Furthermore, $w_{r}:
  \Intdom_{+}^{D} \to \Realdom_{+}$ such that $w_{r}(x)$ is finite for
  all finite arguments $x$.
\end{assumption}

To state the chemical master equation (CME), let for brevity $p(x,t) =
\Prob(X(t) = x | \; X(0) = x_{0})$ be the probability that a certain
number $x$ of molecules is present at time $t$ conditioned upon an
initial state $x_{0}$. The CME is then given by
\cite[Chap.~V]{VanKampen}
\begin{align}
  \label{eq:Master}
  \frac{\partial p(x,t)}{\partial t} &= 
  \sum_{r = 1}^{R} w_{r}(x+\stoich_{r})p(x+\stoich_{r},t)-
  w_{r}(x)p(x,t) =: \Master p(x,t).
\end{align}
The convention of the transpose of the operator to the right of
\eqref{eq:Master} is the standard mathematical formulation of
\emph{Kolmogorov's forward differential system}
\cite[Chap.~8.3]{BremaudMC} in terms of which $\Masteradj$ is the
\emph{infinitesimal generator} of the associated Markov process. This
is also the adjoint of the \emph{master operator} $\Master$ in the
sense that $(\Master p,q) = (p,\Masteradj q)$ in the Euclidean inner
product over the state space. An explicit representation is
\begin{align}
  \label{eq:Masteradj}
  \Masteradj q(x) &= \sum_{r = 1}^{R} w_{r}(x) [q(x-\stoich_{r})-q(x)],
  \intertext{such that the propensities in \eqref{eq:prop} can be retrieved,}
  \Masteradj(x,x-\stoich_{r}) &= w_{r}(x).
\end{align}
Under assumptions to be prescribed in Section~\ref{subsec:ass} it
holds that the dynamics of the expected value of some time-independent
unknown function $f$, conditioned upon the initial state $x_{0}$, can
be written
\begin{align}
  \nonumber
  \frac{d}{dt} E^{x_{0}} [f(X_{t})] &=
  \sum_{x \in \Intdom_{+}^{D}} \frac{\partial p(x,t)}{\partial t} f(x) = 
  (\Master p,f) = \\
  \label{eq:dDynkin}
  &= (p,\Masteradj f) = \sum_{r = 1}^{R}
  E^{x_{0}} \left[ w_{r}(X_{t}) \left( f(X_{t}-\stoich_{r})-f(X_{t}) \right) \right].
\end{align}
We now consider a path-wise representation for the stochastic process
$X_{t}$.

\subsection{The sample path representation}

In the present context of analyzing models in stochastic chemical
kinetics, the path-wise jump SDE representation seems to have been
first put to use in \cite[\textit{manuscript}]{ME2SDE}, and it was
later further detailed in \cite{tau_li}. It should be noted, however,
that an equivalent representation was used much earlier by Kurtz (see
the monograph \cite{Markovappr}).

We thus assume the existence of a probability space
$(\Probspace,\Probfiltr,\Prob)$ with the filtration $\Probfiltr_{t \ge
  0}$ containing $R$-dimensional Poisson processes. The state of the
system $X(t) \in \Intdom_{+}^{D}$ will be constructed from a
stochastic integral with respect to suitably chosen Poisson random
measures.

The transition probability \eqref{eq:prop} defines a \emph{counting
  process} $\pi_{r}(t)$ counting at time $t$ the number of reactions
of type $r$ that has occurred since $t = 0$. It follows that these
processes fully determine the state $X(t)$,
\begin{align}
  \label{eq:Poissrepr}
  X_{t} &= X_{0}-\sum_{r = 1}^{R} \stoich_{r} \pi_{r}(t).
\end{align}
The counting processes are obtained from the transition intensities
(cf.~\eqref{eq:prop})
\begin{align}
  \Prob[\pi_{r}(t+dt)-\pi_{r}(t) = 1| \; \Probfiltr_{t}] &=
  w_{r}(X_{t-})\,dt+o(dt),
\end{align}
where by $X(t-)$ we mean the value of the process prior to any
transitions occurring at time $t$, and where the little-o notation is
understood uniformly with respect to the state
variable. Alternatively, using Kurtz's \emph{random time change
  representation} \cite[Chap.~6.2]{Markovappr}, we can produce the
counting process from a standard unit-rate Poisson process $\Pi_{r}$,
\begin{align}
  \label{eq:count}
  \pi_{r}(t) &= \Pi_{r} \left( \int_{0}^{t} w_{r}(X_{s-}) \, ds \right).
\end{align}

The \emph{marked} counting measure \cite[Chap.~VIII]{pointPQ}
$\mu_{r}(dt \times dz; \, \Probelem)$ with $\Probelem \in \Probspace$
defines an increasing sequence of arrival times $\tau_{i} \in
\Realdom_{+}$ with corresponding ``marks'' $z_{i} \in \markspace :=
[0,1]$ according to some probability distribution which we will take
to be uniform. The intensity $m_{r}(dt \times dz)$ of $\mu_{r}(dt
\times dz)$ is the Lebesgue measure scaled by the corresponding
propensity, $m_{r}(dt \times dz) = w_{r}(X_{t-}) \, dt \times
dz$. Using this formalism, \eqref{eq:Poissrepr} and \eqref{eq:count}
can be written in the jump SDE form
\begin{align}
  \label{eq:SDE0}
  dX_{t} &= -\int_{\markspace} \stoich \boldsymbol{\mu}(dt \times dz),
\end{align}
where $\boldsymbol{\mu} = [\mu_{1},\ldots,\mu_{R}]^{T}$. Here, the
time $\tau-t$ to the arrival of the next reaction of type $r$ is
exponentially distributed with intensity $w_{r}(X_{t-})$. Note that,
by virtue of the nature of the propensities, the intensities of the
counting processes therefore depend nonlinearly on the state
\cite[Chap.~II.3]{pointPQ}.

Using that the point processes are independent and therefore have no
common jump times \cite[Chap.~8.1.3]{BremaudMC}, we can obtain a
sometimes more transparent notation in terms of a \emph{scalar}
counting measure. Define for this purpose and for any state $x$ the
cumulative intensities
\begin{align}
  \label{eq:cumulative}
  W_{r}(x) &= \sum_{s = 1}^{r} w_{s}(x),
\end{align}
such that the total intensity is given by $W(x) \equiv W_{R}(x)$. Let
the marks $z_{i}$ be uniformly distributed on $\markspace$. Then the
frequency of each reaction can be controlled through a set of
indicator functions $\hatw_{r} : \Intdom_{+}^{D} \times \markspace \to
\{0,1\}$ defined according to
\begin{align}
  \label{eq:indicator}
  \hatw_{r}(x; \, z) &= \left\{ \begin{array}{l}
    1 \quad \mbox{ if } W_{r-1}(x) < z W(x) \le W_{r}(x), \\
    0 \quad \mbox{ otherwise.}
  \end{array} \right.
\end{align}
Put $\hatw(x) \equiv [\hatw_{1}(x; \, z), \ldots, \hatw_{R}(x; \,
z)]^{T}$ and define also for later use the indicator form
\begin{align}
  \hatF(x; \, z) &= -\stoich \hatw(x; \, z),
  \intertext{such that}
  F(x) &= \int_{\markspace} \hatF(x; \, z) W(x) \, dz,
\end{align}
where $F(x)$ is defined in \eqref{eq:ODE}.

The jump SDE \eqref{eq:SDE0} can now be written in terms of a scalar
counting random measure $\mu$ through a state-dependent \emph{thinning
  procedure} \cite[Chap.~7.5]{point_processes},
\begin{align}
  \label{eq:SDE}
  dX_{t} &= -\int_{\markspace}
  \stoich \hatw(X_{t-}; \, z) \, \mu(dt \times dz).
\end{align}
Eq.~\eqref{eq:SDE} expresses exponentially distributed reaction times
that arrive according to a point process of intensity $m(dt \times dz)
= W(X_{t-}) \, dt \times dz$ carrying a mark which is uniformly
distributed in $\markspace$. This mark implies the ignition of one of
the reaction channels according to the acceptance-rejection rule
\eqref{eq:indicator}.

One frequently decomposes \eqref{eq:SDE} into its ``drift'' and
``jump'' parts,
\begin{align}
  \label{eq:SDEsplit}
  dX_{t} &= -\stoich w(X_{t}) \, dt-
  \int_{\markspace} \stoich \hatw(X_{t-}; \, z)(\mu-m)(dt \times dz).
\end{align}
The second term in \eqref{eq:SDEsplit} is driven by the compensated
measure $(\mu-m)$ and is a local martingale provided in essence that
the path is absolutely integrable (see \cite[Chap.~VIII.1, Corollary
  C4]{pointPQ} for details).


\subsubsection{Localization; It\^o's and Dynkin's formulas}

In analytic work it is often necessary to `tame' the process by
deriving results under a stopping time $\tau_{\stopping} := \inf_{t
  \ge 0}\{\|X_{t}\| > \stopping\}$ in some norm. Results for the
stopped process $X_{t \wedge \tau_{\stopping}}$ can then be
transferred to the original process by letting $\stopping \to \infty$
under suitable conditions.

Although there are many general versions of It\^o's change of
variables formula available in the setting of semi-martingales (see
for example \cite[Chap.~2.7]{jumpSDEs} and
\cite[Chap.~II.7]{protterSDE}), we shall get around with the following
simple version \cite[Chap.~4.4.2]{LevySDEs}. By the properties of the
semi-martingale pure jump process we have for $\hat{t} = t \wedge
\tau_{\stopping}$
\begin{align}
  \label{eq:iIto}
  f(X_{\hat{t}})-f(X_{0}) &= \sum_{0 < s \le \hat{t}} f(X_{s})-f(X_{s-})
  =\int_{0}^{\hat{t}} \int_{\markspace} f(X_{s})-f(X_{s-}) \, \mu(ds \times dz),
\end{align}
where the sum is over jump times $s \in (0,\hat{t}]$. Using that
  $X_{s} = X_{s-}-\stoich \hatw(X_{s-}; \, z)$ we can write this in
  differential form as
\begin{align}
  \label{eq:Ito}
  df(X_{t}) &= \int_{\markspace} f(X_{t-}-\stoich \hatw(X_{t-}; \, z))-
  f(X_{t-}) \, \mu(dt \times dz).
\end{align}
Alternatively, decomposing \eqref{eq:iIto} into drift- and jump parts
and taking expectation values we get, since the compensated measure is
a local martingale,
\begin{align}
  \nonumber
  Ef(X_{\hat{t}})-Ef(X_{0}) &= E \, \int_{0}^{\hat{t}} \int_{\markspace}
  f(X_{s-}-\stoich \hatw(X_{s-}; \, z))-
  f(X_{s-}) \, m(ds \times dz) \\
  \nonumber
  &= E \, \int_{0}^{\hat{t}} \int_{\markspace}
  \left[ f(X_{s-}-\stoich \hatw(X_{s-}; \, z))-
  f(X_{s-}) \right] W(X_{s-}) \, ds \times dz \\
  \label{eq:Dynkin}
  &= E \, \int_{0}^{\hat{t}} \sum_{r = 1}^{R} \left[ (f(X_{s}-\stoich_{r})-
  f(X_{s})) w_{r}(X_{s}) \right] \, ds.
\end{align}
This is Dynkin's formula \cite[Chap.~9.2.2]{BremaudMC} for the stopped
process and we note that \eqref{eq:dDynkin} is just a differential version.

\subsubsection{Coupled processes}
\label{subsubsec:coupling}

When considering stability properties we will need to compare
different trajectories with respect to the same noise. The details of
this coupling is not defined in either \eqref{eq:SDE0} or
\eqref{eq:SDE} and must in fact be chosen explicitly. Since this
equality is easy to inspect for a unit-rate Poisson process, the
viewpoint of \emph{local time} expressed in \eqref{eq:count} provides
an answer; two processes $X_{t}$ and $Y_{t}$ may be regarded as
coupled if and only if they are evolved using identical Poisson
processes $\Pi_{r}$, $r = 1,\ldots, R$ in \eqref{eq:Poissrepr} and
\eqref{eq:count}. This approach was first used by Kurtz
\cite{KurtzApprox} in the context of the random time change
representation. Algorithmically it implies the \emph{Common Reaction
  Path} (CRP) method for simulating coupled processes
\cite{sensitivitySSA} (see also \cite{ssa_parareal}).

A refinement of this construction was devised, also by Kurtz, in
\cite[\textit{(see Eqs.~(2.2)--(2.3))}]{countingProcesses}. In turn,
this approach implies the \emph{Coupled Finite Difference} method
\cite{parameterCTMC} (but see also \cite{tau_li,ME2SDE}), and is more
amenable to analysis. This is also the construction formalized below
under our current framework.

To obtain such a coupled version of \eqref{eq:SDE} we will have to
make the thinning dependent on both trajectories. This is achieved by
firstly replacing the cumulative intensities in \eqref{eq:cumulative}
with the \emph{base} (or \emph{minimal}) intensities
\begin{align}
  W_{r}^{(0)}(x,y) &= \sum_{s = 1}^{r} w_{s}(x) \wedge w_{s}(y),
\end{align}
and use the new total base intensity $W^{(0)}(x,y) \equiv
W_{R}^{(0)}(x,y)$ as the intensity of the counting measure $\mu_{0}$;
$m_{0}(dt \times dz) = W^{(0)}(X_{t-},Y_{t-}) \, dt \times dz$. We
also modify \eqref{eq:indicator} accordingly,
\begin{align}
  \hatw_{r}^{0}(x,y; \, z) &= \left\{ \begin{array}{l}
    1 \quad \mbox{ if } W_{r-1}^{(0)}(x,y) < z W^{(0)}(x,y)
    \le W_{r}^{(0)}(x,y), \\
    0 \quad \mbox{ otherwise.}
  \end{array} \right.
\end{align}

Secondly, we also define the \emph{remainder} intensity,
\begin{align}
  W_{r}^{(\delta)}(x,y) &= \sum_{s = 1}^{r} w_{s}(x) \vee
  w_{s}(y)-W_{r}^{(0)}(x,y) = \sum_{s = 1}^{r} |w_{s}(x)-w_{s}(y)|.
\end{align}
In analogy with the previous construction we have the associated total
intensity $W^{(\delta)}(x,y) \equiv W_{R}^{(\delta)}(x,y)$ and
counting measure $\mu_{\delta}$; $m_{\delta}(dt \times dz) =
W^{(\delta)}(X_{t-},Y_{t-}) \, dt \times dz$. This time the thinning
procedure is non-symmetric in its two first arguments,
\begin{align}
  \hatw_{r}^{\delta}(x,y; \, z) &= \left\{ \begin{array}{l}
    1 \quad \mbox{ if } W_{r-1}^{(\delta)}(x,y) < z W^{(\delta)}(x,y)
    \le W_{r-1}^{(\delta)}(x,y)+d(x,y), \\
    0 \quad \mbox{ otherwise,}
  \end{array} \right. \\
\intertext{with the non-symmetricity due to}
   d(x,y) &= w_{r}(x)-w_{r}(x) \wedge w_{r}(y).
\end{align}

As a concrete example of how this comparative thinning might be used,
consider the following variant of \eqref{eq:Ito},
\begin{align}
  \nonumber
  df(X_{t}-Y_{t}) = \int_{\markspace} &f\left( X_{t-}-\stoich
  \hatw^{\delta}(X_{t-},Y_{t-}; \, z)-Y_{t-}+\stoich
  \hatw^{\delta}(Y_{t-},X_{t-}; \, z) \right)- \\
  \label{eq:diff_f1}
  &f(X_{t-}-Y_{t-}) \, \mu_{\delta}(dt \times dz).
\end{align}
For this specific example, the terms governed by the base counting
measure $\mu_{0}$ cancel out altogether.

We mention also that an equivalent construction, but one that leads to
different algorithms, can be obtained via a thinning of a single
measure \cite{tau_li,ME2SDE}. Defining instead
\begin{align}
  W_{r}^{(+)}(x,y) &= \sum_{s = 1}^{r} w_{s}(x) \vee w_{s}(y),
\end{align}
implying the total intensity $W^{(+)}(x,y) \equiv W_{R}^{(+)}(x,y)$
and associated counting measure $\mu_{+}$; $m_{+}(dt \times dz) =
W^{(+)}(X_{t-},Y_{t-}) \, dt \times dz$. By construction the indicator
functions are now non-symmetric in their first two arguments,
\begin{align}
  \hatw_{r}^{+}(x,y; \, z) &= \left\{ \begin{array}{l}
    1 \quad \mbox{ if } W_{r-1}^{(+)}(x,y) < z W^{(+)}(x,y)
    \le W_{r-1}^{(+)}(x,y)+w_{r}(x), \\
    0 \quad \mbox{ otherwise.}
  \end{array} \right.
\end{align}
In analogy to \eqref{eq:diff_f1} we get
\begin{align}
  \nonumber
  df(X_{t}-Y_{t}) = \int_{\markspace} &f\left( X_{t-}-\stoich
  \hatw^{+}(X_{t-},Y_{t-}; \, z)-Y_{t-}+\stoich
  \hatw^{+}(Y_{t-},X_{t-}; \, z) \right)- \\
  \label{eq:diff_f2}
  &f(X_{t-}-Y_{t-}) \, \mu_{+}(dt \times dz).
\end{align}
This time, however, the intensity of the counting measure is generally
larger and the equivalence is obtained as a result of the thinning
procedure.

\subsubsection{The validity of the master equation}

With this much formalism developed, we may conveniently quote the
following result:
\begin{theorem}[\textit{\cite[Chap.~8.3.2, Theorem~3.3]{BremaudMC}}]
  \label{th:mastervalid}
  Under Assumption~\ref{ass:prop0}, and if additionally, for $t \in
  [0,T]$ it holds that
  \begin{align}
    E \, W(X_{t}) &< \infty,
  \end{align}
  then \eqref{eq:Master} is valid for $t \in [0,T]$.
\end{theorem}

Since the governing equation \eqref{eq:dDynkin} for the expected value
of $f(X_{t})$ is a direct consequence of \eqref{eq:Master}, we can
similarly conclude the following:
\begin{corollary}
  \label{cor:mastervalid}
  Under the assumptions of Theorem~\ref{th:mastervalid}, and if,
  moreover, in an arbitrary norm $\|\cdot\|$,
  \begin{align}
    \label{eq:finitef}
    E \, \|f(X_{t})\| &< \infty, \qquad t \in  [0,T],
  \end{align}
  then \eqref{eq:dDynkin} is valid for $t \in [0,T]$.
\end{corollary}

In stating these results we have suppressed the conditional dependency
on the initial state which we for simplicity consider to be some
non-random state $x_{0}$.

\subsection{Concrete examples}

Consider the bi-molecular birth-death system,
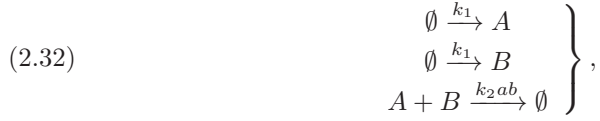
\begin{align}
  \label{eq:bimol}
  \left. \begin{array}{c}
    \emptyset \xrightarrow{k_{1}} A \\
    \emptyset \xrightarrow{k_{1}} B \\
    A+B \xrightarrow{k_{2} ab} \emptyset
  \end{array} \right\},
\end{align}
that is, the system is in contact with a large reservoir such that
$A$- and $B$-molecules are emitted at a constant rate
$k_{1}$. Additionally, a \emph{decay} reaction happens with
probability $k_{2}$ per unit of time whenever two molecules meet. For
this example we have the stoichiometric matrix
\begin{align*}
  \stoich &= \begin{bmatrix*}[r]
    -1 & 0 & 1 \\
    0 & -1 & 1
  \end{bmatrix*}
  \intertext{and the vector propensity function}
  w(x) &= [k_{1},k_{1},k_{2}ab]^{T},
\end{align*}
where $x = [x_{1},x_{2}]^{T} = [a,b]^{T}$.

For a state $X_{t} = [A_{t},B_{t}]^{T}$, define $U_{t} = A_{t}-B_{t}
\in \Intdom$. It\^o's formula \eqref{eq:Ito} with $f(x) = x_{1}-x_{2}$
yields
\begin{align}
  dU_{t} &= df(X_{t}) = \int_{I} -[-1,1,0]
  \hatw(X_{t-}; \, z) \, \mu(dt \times dz),
\end{align}
which upon a moments consideration is just the same thing as the model
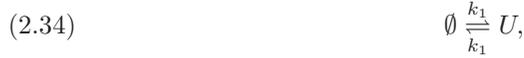
\begin{align}
  \emptyset \overset{k_{1}}{\underset{k_{1}}{\rightleftharpoons}} U,
\end{align}
that is, a constant intensity discrete random walk process. An
explicit solution is the difference between two independent Poisson
distributions,
\begin{align}
  U_{t} &= U_{0}+\Pi_{1}(k_{1}t)-\Pi_{2}(k_{1}t)
  \sim \mathcal{N}(U_{0},2k_{1}t), \quad \mbox{as $t \to \infty$,}
\end{align}
where $\mathcal{N}$ is a normally distributed random variable of the
indicated mean and variance. Hence $U_{t}$ fluctuates between
arbitrarily large and small values as $t \to \infty$.

\subsubsection{Reversible versions}

From time to time below we shall be concerned with the following
\emph{closed} version of \eqref{eq:bimol}, consisting of a single
reversible reaction,
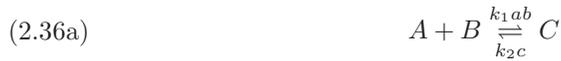
\begin{subequations}
\begin{align}
  \label{eq:reversible}
  \begin{array}{c}
    A+B \overset{k_{1}ab}{\underset{k_{2}c}{\rightleftharpoons}} C
  \end{array}
\end{align}
This is clearly a finite system since the number $a+b+2c$ is always
preserved. An \emph{open} version of the same system is
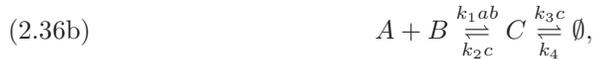
\begin{align}
  \label{eq:reversibleopen}
  \begin{array}{c}
    A+B \overset{k_{1}ab}{\underset{k_{2}c}{\rightleftharpoons}} C
    \overset{k_{3}c}{\underset{k_{4}}{\rightleftharpoons}} \emptyset,
  \end{array}
\end{align}
\end{subequations}
and will prove to be a useful example in the stochastic setting since
formally, all states in $\Intdom_{+}^{3}$ are reachable. For
\eqref{eq:reversible} we have
\begin{align}
  \stoich &\equiv \begin{bmatrix*}[r]
      1 & -1 \\
      1 & -1 \\
      -1 & 1
  \end{bmatrix*}, \quad
  w(x) \equiv [k_{1}ab,k_{2}c]^{T}, \\
  \intertext{while \eqref{eq:reversibleopen} is represented by}
  \stoich &\equiv \begin{bmatrix*}[r]
      1 & -1 & 0 & 0 \\
      1 & -1 & 0 & 0 \\
      -1 & 1 & 1 & -1
  \end{bmatrix*}, \quad
  w(x) \equiv [k_{1}ab,k_{2}c,k_{3}c,k_{4}]^{T}.
\end{align}
These examples, while very simple to deal with, will provide good
counterexamples in both Section~\ref{sec:det} and \ref{sec:stoch}.


\section{Deterministic stability}
\label{sec:det}

In this section we shall be concerned with the deterministic drift
part of the dynamics \eqref{eq:SDEsplit}. We are interested in
techniques for judging the stability of the time-homogeneous ODE
\eqref{eq:ODE}, the so-called reaction rate equations implied by the
rates \eqref{eq:prop}. Stability and continuity with respect to
initial data are considered in Sections~\ref{subsec:stab} and
\ref{subsec:cont}. The main motivation for this discussion stems from
the observation that assumptions that \emph{do not} hold in this very
basic setting are unlikely to hold in the stochastic case. In
Section~\ref{subsec:constants}, techniques for explicitly obtaining
all our postulated constants are discussed. A good point in favor of
taking the time to describe these techniques is that we have not found
such a discussion elsewhere.

Initially we will consider states $x \in \Realdom^{D}$, but we will
soon find it convenient to restrict the treatment to $x \in
\Realdom_{+}^{D}$. In order to remain valid also in the discrete
stochastic setting, however, constructed counterexamples will remain
relevant also when restricted to $\Intdom_{+}^{D}$.

\subsection{Stability}
\label{subsec:stab}


Many stability proofs can be thought of as comparisons with relevant
linear cases. This is the motivation for the well-known Grönwall's
inequality which we state in the following two versions.
\begin{lemma}
  \label{lem:gronwall}
  Suppose that $u'(t) \le A+\alpha u(t)$ for $t \ge 0$. Then
  \begin{align}
    u(t) &\le u(0) e^{\alpha t}+\frac{A}{\alpha}
    \left( e^{\alpha t}-1 \right). \\
    \intertext{The same conclusion holds irrespective of the
      differentiability of $u$ but with $\alpha \ge 0$ and under the
      weaker integral condition}
    u(t) &\le u(0)+\int_{0}^{t} A+\alpha u(s) \, ds.
  \end{align}
\end{lemma}


The most immediate way of comparing the growth of solutions to the ODE
\eqref{eq:ODE} to those of a linear ODE is to require that the norm of
the driving function is bounded in terms of its argument;
\begin{align}
  \label{eq:absbndreq}
  \|F(x)\| &\le A+\alpha \|x\|,
  \intertext{since then by the triangle inequality,}
  \|x(t)\| \le \|x_{0}\|+\int_{0}^{t} \|F(x)\| \, dt &\le
  \|x_{0}\|+\int_{0}^{t} A+\alpha \|x(t)\| \, dt,
\end{align}
where Grönwall's inequality applies. Unfortunately,
\eqref{eq:absbndreq} is a too strict requirement for our applications.

\begin{proposition}
  \label{prop:bimol1}
  The bi-molecular birth-death system \eqref{eq:bimol} does not
  satisfy \eqref{eq:absbndreq}.
\end{proposition}

\begin{proof}
  We compute $\|F(x)\| = \|-\stoich w(x) \| = \sqrt{2}|k_{1}-k_{2}ab|$
  for a state $x = [a,b]^{T}$. Hence for $a = b = N = 0,1, \ldots$ we
  have for $N$ large enough that $\|F(x)\| = \sqrt{2} k_{2} \cdot
  N^{2}-\sqrt{2}k_{1}$, which can clearly never be bounded linearly in
  $\|x\| = \sqrt{2} N$.
\end{proof}

The problem with the simple condition \eqref{eq:absbndreq} is that it
does not take the direction of growth into account; the offending
quadratic propensity in \eqref{eq:bimol} actually \emph{decreases} the
number of molecules. To deal with this, let $x \in \Realdom^{D}$ be an
arbitrary vector defining an ``outward'' direction. The length of the
component of the driving function along this direction is $(x,F(x))$
and in order not to have $x$ driven too strongly out along this ray we
may, in view of Grönwall's inequality, naturally require that
$(x/\|x\|,F(x)) \le \mbox{constant} \times \|x\|$ for $\|x\|$
sufficiently large. Equivalently, for any $x$,
\begin{align}
  \label{eq:bndreq}
  (x,F(x)) &\le A+\alpha \|x\|^{2},
  \intertext{from which one deduces}
  \frac{d}{dt} \frac{\|x\|^{2}}{2} &= (x,F(x)) \le A+\alpha \|x\|^{2},
\end{align}
where Grönwall's inequality applies anew. The assumption
\eqref{eq:bndreq} is weaker than \eqref{eq:absbndreq} since the former
implies the latter by the Cauchy-Schwarz inequality. Indeed, as in the
proof of Proposition~\ref{prop:bimol1} it is readily checked that for
the bi-molecular birth-death system \eqref{eq:bimol}, we get $(x,F(x))
= k_{1}(a+b)-k_{2}(a+b)ab$ which this time readily can be bounded
linearly in terms of $\|x\|^{2} = a^{2}+b^{2}$.

Unfortunately, in the case of an infinite state space and strong
dependencies between the species the assumption \eqref{eq:bndreq} is
also often unrealistic.

\begin{proposition}
  \label{prop:reversible1}
  Neither \eqref{eq:reversible} nor \eqref{eq:reversibleopen} admits a
  bound of the kind \eqref{eq:bndreq}.
\end{proposition}

\begin{proof}
  As in the proof of Proposition~\ref{prop:bimol1} we look at a ray
  $x^{T} = (a,b,c) = (N,N,3N)$ parametrized by a non-negative integer
  $N$. For \eqref{eq:reversible} we compute $(x,F(x)) = (x,-\stoich
  w(x)) = (a+b-c)(k_{2}c-k_{1}ab) = k_{1}N^{3}-3k_{2}N^{2}$, which
  clearly cannot be bounded linearly in $\|x\|^{2} = 11N^{2}$. The
  same argument applies also to \eqref{eq:reversibleopen}.
\end{proof}

This negative result can perhaps best be appreciated as a kind of loss
of information about the dependencies between the species in the
functional form of the condition \eqref{eq:bndreq}. The number of $A$-
and $B$-molecules is strongly correlated with the number of
$C$-molecules such that, in fact, in \eqref{eq:reversible} $a+b+2c$ is
a preserved quantity. By contrast, in \eqref{eq:bndreq} the growth of
$\|x\|^{2}$ is estimated from the sum of the growth of the individual
elements of $x$ as if they where independent.

A way around this limitation can be found provided that we leave the
general case $x \in \Realdom^{D}$. We therefore specify the discussion
to the positive quadrant $x \in \Realdom_{+}^{D}$ and assume from now
on that it can be shown \textit{a priori} that the initial data
$x_{0}$ belongs to this set and that the subsequent trajectory $x(t)$
never leaves it (compare Assumption~\ref{ass:prop0}).

It then follows that $\|x_{t}\|_{1} = (\onevec,x_{t}) = \onevect
x_{t}$, where $\onevec$ is the vector of length $D$ containing all
ones. This vector also defines a suitable ``outward'' vector for
states $x \in \Realdom_{+}^{D}$ since solutions to the ODE
\eqref{eq:ODE} cannot grow without simultaneously growing also in the
direction of $\onevec$.

Again, in view of Grönwall's inequality Lemma~\ref{lem:gronwall}, we
tentatively require that $(\onevec,F(x)) \le \mbox{constant} \times
\|x\|_{1}$ for $\|x\|_{1}$ sufficiently large. Equivalently, for any
$x$,
\begin{align}
  \label{eq:onereq}
  (\onevec,F(x)) &\le A+\alpha \|x\|_{1},
  \intertext{implying the bounded dynamics}
  \frac{d}{dt} \|x\|_{1} &= (\onevec,F(x)) \le A+\alpha \|x\|_{1}.
\end{align}
We remark in passing that the criterion \eqref{eq:onereq} is sharp in
the sense that if the reversed inequality can be shown to be true,
then the growth of solutions can be estimated from below.

\begin{example}
  \label{ex:favor}
  As a point in favor of this approach we compute for the bi-molecular
  birth-death system \eqref{eq:bimol}, $(\onevec,F(x)) =
  (\onevec,-\stoich w(x)) = 2k_{1}-2k_{2}ab$ which evidently falls
  under the assumption \eqref{eq:onereq} with $(A,\alpha) =
  (2k_{1},0)$. For the reversible case \eqref{eq:reversible} we
  similarly get $(\onevec,F(x)) = -k_{1}ab+k_{2}c$ such that
  \eqref{eq:onereq} applies with $(A,\alpha) = (0,k_{2})$. Finally,
  and in the same fashion, the open case \eqref{eq:reversibleopen} is
  seen to be covered by letting $(A,\alpha) = (k_{4},(k_{2}-k_{3})
  \vee 0)$.
\end{example}

The chosen ``outward'' vector $\onevec$ is by no means
special. Clearly, any strictly positive vector $\lvec$ may be used in
its place since $\|x\|_{1}$ and $\lnorm{x} := \lvect x$ are equivalent
norms over $\Realdom_{+}^{D}$. This is a general and useful
observation as it may be used to discard parts of a system that are
closed without any restrictions on the associated propensities.

%

\begin{example}
  \label{ex:reversiblel}
  For the reversible system \eqref{eq:reversible}, we have already
  noted that $a+b+2c$ is a conserved quantity such that the choice
  $\lvec = [1,1,2]^{T}$ yields $d/dt \lnorm{x} = 0$. The open case
  \eqref{eq:reversibleopen} also benefits from this weighted norm in
  that we get $d/dt \lnorm{x} \le 2k_{4}$.
\end{example}

\begin{example}
  \label{ex:quadraticl}
  A slightly more involved model reads as follows:
  \begin{align*}
    \left. \begin{array}{rclcrcl}
        \emptyset & \overset{k_{1}}{\rightarrow} & A & &
        A+B & \overset{k_{2}ab}{\rightarrow} & 3C \\
        \emptyset & \overset{k_{1}}{\rightarrow} & B & &
        C & \overset{k_{3}c}{\rightarrow} & \emptyset
      \end{array} \right\}
  \end{align*}
  This example has been constructed such that the quadratic reaction
  increases $\|x\|_{1}$ and hence \eqref{eq:onereq} does not
  apply. However, taking $\lvec = [3,3,2]^{T}$ we get
  \begin{align*}
    \frac{d}{dt} \lnorm{x_{t}} &= 6k_{1}-2k_{3}c \le 6k_{1}.
  \end{align*}
\end{example}

This example hints at a general technique for obtaining suitable
candidates for the weight vector $\lvec$. Simply form the matrix
$\stoich_{2}$ consisting of the columns of $\stoich$ that are affected
by superlinear propensities. If a vector $\lvec > 0$ annihilating
these propensities exists, it can be found in the null-space of
$-\stoich_{2}^{T}$, readily available by linear algebra techniques. We
omit the details.


\subsection{Continuity}
\label{subsec:cont}

For well-posedness of the ODE \eqref{eq:ODE} we also need continuity
with respect to the initial data. We cannot ask for uniform Lipschitz
continuity since $\|F(x)-F(y)\| \le L \|x-y\|$ clearly implies
\eqref{eq:absbndreq} which we have already refuted. For the same
reason, a uniform \emph{one-sided Lipschitz condition}
$(x-y,F(x)-F(y)) \le \lambda \|x-y\|^{2}$ cannot be assumed to hold
since it implies \eqref{eq:bndreq}. The problem here is the global
nature of the estimate and it therefore seems to be reasonable to
localize this assumption. For instance, one might ask for
\begin{align}
  (x-y,F(x)-F(y)) &\le \lambda_{R} \|x-y\|^{2} \mbox{ whenever }
  \|x\|, \, \|y\| \le R, \\
  \intertext{presumably with some growth restrictions on
    $\lambda_{R}$. Although very general, such an analysis is likely
    to be less informative when it comes to estimating actual
    constants in later results. We shall therefore consider the
    following simpler version,}
  \label{eq:nlinreq}
  (x-y,F(x)-F(y)) &\le (M+\mu\|x+y\|_{1}) \|x-y\|^{2},
\end{align}
where the form of $\lambda_{R}$ has been restricted to better suit the
present purposes. Trivially, the norms $\|\cdot\|_{1}$ and $\|\cdot\|$
are equivalent and hence the specific choice made in
\eqref{eq:nlinreq} is just a matter of convenience. Since the idea
here is to use \textit{a priori} bounds on $x$ and $y$ when deriving
perturbation bounds, using $\|\cdot\|_{1}$ (or $\lnorm{\cdot}$) is
natural.

\begin{theorem}
  \label{th:ODEcont}
  Suppose that the ODE \eqref{eq:ODE} satisfies \eqref{eq:onereq} and
  \eqref{eq:nlinreq} and that initial data $x_{0} \in
  \Realdom_{+}^{D}$ implies a solution $x(t) \in
  \Realdom_{+}^{D}$. Then for any $t \in [0,T]$ there is a unique such
  solution $x(t)$. Moreover, define $C(t;\, x_{0},y_{0}) \equiv M+\mu
  \|x_{t}+y_{t}\|_{1}$, where $x_{t}$ and $y_{t}$ are two trajectories
  associated with initial data $x_{0}, \, y_{0} \in \Realdom_{+}^{D}$,
  respectively. Then
  \begin{align}
    \label{eq:ODEcont}
    \|x(t)-y(t)\| &\le \|x_{0}-y_{0}\|
    \exp \left( \int_{0}^{t} C(s; \; x_{0},y_{0}) \, ds \right) \\
    \nonumber
    &= \|x_{0}-y_{0}\| \Big[ 1+
    \big(M+\mu(\|x_{0}+y_{0}\|_{1})\big)t+
    \Ordo{t^{2}} \Big].
  \end{align}
\end{theorem}

\begin{proof}
  Combining \eqref{eq:onereq} with Grönwall's inequality we get the
  \textit{a priori} estimate
  \begin{align*}
    \|x(t)+y(t)\|_{1} &\le \|x_{0}+y_{0}\|_{1}+
    \left( \alpha\|x_{0}+y_{0}\|_{1}+2A \right)
    \left( e^{\alpha t}-1\right)/\alpha.
  \end{align*}
  Hence the (bounded) solution to
  \begin{align*}
    \frac{d}{dt} \|x-y\|^{2} &= 2(x-y,F(x)-F(y)) \le 
    2C(t;\,x_{0},y_{0}) \|x-y\|^{2}.
  \end{align*}
  is readily found through its integrating factor. The order estimate
  is a consequence of the fact that
  \begin{align*}
    \int_{0}^{t} \|x_{s}\|_{1}-\|x_{0}\|_{1} \, ds = \Ordo{t^{2}},
  \end{align*}
  since the trajectory is continuous.
\end{proof}

\subsection{Bounds for elementary reactions}
\label{subsec:constants}

As briefly discussed by the end of Section~\ref{subsec:stab}, finding
bounds on $A$ and $\alpha$ in \eqref{eq:onereq} as well as a suitable
weight-vector $\lvec$ amounts to basic inequalities and some fairly
straightforward linear algebra manipulations. In this section we
therefore consider precise bounds in \eqref{eq:nlinreq} for the
elementary propensities \eqref{eq:el}. Since \eqref{eq:nlinreq} is
linear in $F$, a reasonable approach is to consider linear and
quadratic propensities separate (constant propensities trivially
satisfy \eqref{eq:nlinreq} with $M = \mu = 0 $).

\begin{proposition}[\textit{linear case}]
  \label{prop:linear}
  Write a set of $R$ linear propensities as $w_{r}(x) = q_{r}^{T}x$,
  $r = 1,\ldots,R$, each with the corresponding stoichiometric vector
  $\stoich_{r}$. Then $F(x) := -\sum_{r} \stoich_{r} w_{r}(x)$
  satisfies $(x-y,F(x)-F(y)) \le M \|x-y\|^{2}$ with $M =
  \lognorm{-\stoich Q^{T}}$ in terms of the Euclidean logarithmic norm
  $\lognorm{\cdot}$ and the matrix $Q$ containing the vectors $q_{r}$
  as columns. In particular, in the case of a single linear propensity
  and, if as is usually the case, $q_{j} = k\delta_{jn}$ is all-zero
  except for a single rate constant $k$ in the $n$th position, then
  this reduces to $M = k \, (-\stoich_{r,n}+\|\stoich_{r}\|)/2$.
\end{proposition}

\begin{proof}
  The first assertion is immediate since the smallest such constant
  $M$ by definition is the logarithmic norm (see
  e.g.~\cite{lognorm}). To compute $\lognorm{-\stoich_{r} q_{r}^{T}}$
  when $q$ has the form indicated, we determine the extremal
  eigenvalue of $-(\stoich_{r}q_{r}^{T}+q_{r}\stoich_{r}^{T})/2$. By
  the (signed) scaling invariance of the logarithmic norm we may
  without loss of generality take $k \equiv 1$. The spectral relation
  for an eigenpair $(\lambda,z)$ can be written as
  \begin{align*}
    -&\frac{1}{2} \stoich_{r,j} z_{n} = \lambda z_{j},
    \quad j = 1,2,\ldots, D, \; j \not = n, \\
    -&\frac{1}{2} \stoich_{r,n} z_{n}-\frac{1}{2} \stoich_{r}^{T}z =
    \lambda z_{n}.
  \end{align*}
  For non-zero $\lambda$ the first relation can be solved for
  $z_{j}$. When inserted into the second relation, using that $z_{n}
  \not = 0$ (or otherwise $z = 0$), we get a quadratic equation for
  $\lambda$ with a single extremal root.
\end{proof}

\begin{example}
  The simple special case in Proposition~\ref{prop:linear} is
  generally sharp except for when there are linear reactions affecting
  \emph{all} species considered in the model. For example, in a
  one-dimensional state space, the single decay $A \to \emptyset$ with
  propensity $w_{1}(a) = k a$ allows the optimal value $M = -k$. In
  general $D$-dimensional space, a chain with unit rate constants of
  the form $A_{1} \to A_{2} \to \cdots \to A_{D} \to \emptyset$, or a
  closed loop in which the last transition is replaced with $A_{D} \to
  A_{1}$, both admit bounds $M \le 0$ as an inspection of the
  Gershgorin-discs of $-(\stoich+\stoich^{T})/2$ shows.

  Other than for those special examples, for the most important linear
  cases, Table~\ref{tab:linear} summarizes the bounds as obtained from
  the special case in Proposition~\ref{prop:linear} (with all reaction
  constants normalized to unity).

  \begin{table}[h!]
    \begin{center}
      \begin{tabular}{lp{1cm}l}
        Reaction & & Bound on $M$ \\
        \hline
        $A \to \emptyset$ & & $0$ \\
        $A \to B$ & & $(\sqrt{2}-1)/2$ \\
        $A \to B+C$ & & $(\sqrt{3}-1)/2$
      \end{tabular}
    \end{center}
    \caption{Linear propensities and bounds of $M$ in
      \eqref{eq:nlinreq}.}
    \label{tab:linear}
  \end{table}
\end{example}

\begin{proposition}[\textit{quadratic case}]
  \label{prop:quadratic}
  Write a general quadratic propensity as $w_{r}(x) = x^{T}Sx$ with
  $S$ a symmetric matrix. Then $F_{r}(x) = -\stoich_{r} w_{r}(x)$
  satisfies \eqref{eq:nlinreq} with $M = 0$ and $\mu =
  \|x+y\|_{1}^{-1} \, \lognorm{-\stoich_{r}(x+y)^{T}S} \le
  \|\stoich_{r}\|\|S\|$. For the special case that $S_{ij} = k
  (\delta_{im}\delta_{jn}+\delta_{jm}\delta_{in})/2$ there holds $\mu
  \le k \, \max_{j \in \{m,n\}} (-\stoich_{r,j}+\|\stoich_{r}\|)/4$.
\end{proposition}

\begin{proof}
  Since $S$ is symmetric we have $x^{T}Sx-y^{T}Sy =
  (x+y)^{T}S(x-y)$. Hence an explicit expression for $\mu$ is obtained
  as follows:
  \begin{align*}
    \mu &\le \sup_{x,y \in \Intdom_{+}^{D}} \|x+y\|_{1}^{-1}\|x-y\|^{-2}
    (x-y,-\stoich_{r}(x+y)^{T}S(x-y)) \\
    &= \sup_{w \in \Intdom_{+}^{D}} \sup_{v \in \Intdom^{D}} 
    \|w\|_{1}^{-1}\|v\|^{-2}(v,-\stoich_{r}w^{T}Sv) \le 
    \sup_{\natop{u \ge 0}{\|u\|_{1} \le 1}} \lognorm{-\stoich_{r}u^{T}S}.
  \end{align*}
  The indicated upper bound is derived from the fact that
  $|\lognorm{\cdot}| \le \|\cdot\|$ \cite{lognorm}. For the useful
  special case, define first the vector $q = q_{1}+q_{2}$ in terms of
  $q_{1,j} = k \delta_{jm}(x_{n}+y_{n})/2$, and $q_{2,j} = k
  \delta_{jn}(x_{m}+y_{m})/2$. Using the fact that the logarithmic
  norm is sub-additive we can reuse the calculation in the proof of
  Proposition~\ref{prop:linear},
  \begin{align*}
    &\lognorm{-\stoich_{r} (x+y)^{T}S} = \lognorm{-\stoich_{r}q^{T}}
    \le \lognorm{-\stoich_{r}q_{1}^{T}}+
    \lognorm{-\stoich_{r}q_{2}^{T}} \\
    &\phantom{\lognorm{}}=
    (x_{n}+y_{n})k (-\stoich_{r,m}+\|\stoich_{r}\|)/4+
    (x_{m}+y_{m})k (-\stoich_{r,n}+\|\stoich_{r}\|)/4.
  \end{align*}
\end{proof}


\begin{example}
  The most important quadratic cases are summarized in
  Table~\ref{tab:quadratic}. For the dimerizations in the lower half
  of the table there is also a linear part $M$ in \eqref{eq:nlinreq}.

  \begin{table}[h!]
    \begin{center}
      \begin{tabular}{lp{0.5cm}ll}
        Reaction & & Bound on $\mu$ & $M$ \\
        \hline
        $A+B \to \emptyset$ & & $(\sqrt{2}-1)/4$ & \\
        $A+B \to C$ & & $(\sqrt{3}-1)/4$ & \\
        $A+B \to A$ & & $1/4$ & \\
        $A+B \to A+C$ & & $\sqrt{2}/4$ & \\
        \hline
        $A+A \to \emptyset$ & & $0$ & $2$ \\
        $A+A \to B$ & & $\sqrt{5}/2-1$ & $\sqrt{5}/2+1$
      \end{tabular}
    \end{center}
    \caption{Quadratic propensities and bounds of $M$, $\mu$ in
      \eqref{eq:nlinreq}.}
    \label{tab:quadratic}
  \end{table}
\end{example}

\medskip

\begin{example}
  The bi-molecular birth-death model \eqref{eq:bimol} admits the
  constants $(M,\mu) = (0,k_{2}(\sqrt{2}-1)/4)$ in
  \eqref{eq:nlinreq}. Similarly, the reversible cases
  \eqref{eq:reversible} and \eqref{eq:reversibleopen} both obeys
  \eqref{eq:nlinreq} with $(M,\mu) = (\sqrt{3}-1)/2 \times
  (k_{2},k_{1}/2)$. All these results are sharp except for the open
  case \eqref{eq:reversibleopen} for which one can obtain a slightly
  smaller constant $M$ by using the general formula stated in
  Proposition~\ref{prop:linear}.
\end{example}

\begin{example}
  As a highly prototypical example we consider the following natural
  extension of the bi-molecular birth-death model \eqref{eq:bimol},
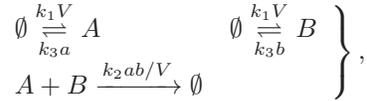
  \begin{align*}
    \left. \begin{array}{ll}
      \emptyset \overset{k_{1}\Vol}{\underset{k_{3}a}{\rightleftharpoons}} A &
      \emptyset \overset{k_{1}\Vol}{\underset{k_{3}b}{\rightleftharpoons}} B \\
      A+B \xrightarrow{k_{2} ab/\Vol} \emptyset & \\
    \end{array} \right\},
  \end{align*}
  where in this example it is informative to consider the dependence
  on the system's size $\Vol$.  It is straightforward to show the
  bounds $(A,\alpha) \le (2k_{1}\Vol,-k_{3})$ in \eqref{eq:onereq} and
  hence that the system is effectively bounded despite being of open
  character. This is seen from the fact that, for states $\|x\|_{1}
  \ge 2k_{1}/k_{3} \cdot \Vol$, the dynamics is \emph{dissipative} in
  the $\|\cdot\|_{1}$-norm. Furthermore, from
  Proposition~\ref{prop:linear} and \ref{prop:quadratic} we get the
  sharp bounds $(M,\mu) \le (-k_{1}\Vol,k_{2}/\Vol \cdot
  (\sqrt{2}-1)/4)$ in \eqref{eq:nlinreq}. It follows that for states
  $\{x,y\}$ such that $\|x+y\|_{1} \lesssim 9.7 k_{1}/k_{2} \cdot
  \Vol^{2}$, the dynamics is \emph{contractive} in the Euclidean
  norm. For density dependent propensities we expect that $\|x\| \sim
  \Vol$ in any norm as $\Vol$ grows, and hence the region of
  contractivity grows in a relative sense. Intuitively one expects
  that these results offer an insight into the evolution of the
  process that is relevant also in the stochastic setting.
\end{example}


\section{Stochastic stability}
\label{sec:stoch}

We now consider the properties of the stochastic jump SDE
\eqref{eq:SDE}. For convenience we start by collecting all assumptions
in Section~\ref{subsec:ass}. In the stochastic setting the
requirements for existence and uniqueness are slightly stronger than
in the deterministic case such that the one-sided bound
\eqref{eq:nlinreq} needs to be augmented with an unsigned version,
implying essentially the assumption of at most quadratically growing
propensities. We demonstrate that this assumption is reasonable by
constructing a model involving cubic propensities and with unbounded
second moments. On the positive side we show in
Section~\ref{subsec:moments} that the assumptions are strong enough to
guarantee finite moments of any order during finite time intervals.

We prove existence and uniqueness of solutions to the jump SDE
\eqref{eq:SDE} in Section~\ref{subsec:wellposed}. A sufficient
condition for the existence of asymptotic bounds of the $p$th order
moment is given in Section~\ref{subsec:stability} where we also derive
some stability estimates.

\subsection{Working assumptions}
\label{subsec:ass}

We state formally the set of assumptions on the jump SDE
\eqref{eq:SDE} as follows.
\begin{assumption}
  \label{ass:ass}
  For arguments $x$, $y \in \Intdom_{+}^{D}$, $F(x) := -\stoich w(x)$,
  and weighted norm $\lnorm{x} := \lvect x$ we assume that
  \begin{enumerate}[(i)]
  \item \label{it:1bnd} $-\lvect \stoich w(x) \le A+\alpha \lnorm{x}$
    \textit{(``bounded growth'')},

  \item \label{it:bnd} $(-\lvect \stoich)^{2} w(x)/2 \le B+\beta_{1}
    \lnorm{x}+\beta_{2}\lnorm{x}^{2}$ \textit{(``absolutely bounded
      growth'')}
  \item
    \label{it:lip} $(\onevect \stoich^{2})|w(x)-w(y)| \le L(1+\|x+y\|_{1})\|x-y\|$,
  \item
    \label{it:1lip} $(x-y,F(x)-F(y)) \le (M+\mu\|x+y\|_{1}) \|x-y\|^{2}$.
  \end{enumerate}
  The parameters $\{A,B,\beta_{1},\beta_{2},L\}$ are assumed to be
  positive (with $\beta_{2}$ possibly zero) but we allow also negative
  values of $\{\alpha,M,\mu\}$. The vector $\lvec$ is normalized such
  that $\min_{i} \lvec_{i} = 1$; hence the bound $\|x\|_{1} \le
  \lnorm{x}$ is sharp.
\end{assumption}

After the original draft of the current paper was finished, the author
became aware of two other papers discussing very similar conditions
\cite{ergodic_jsde,momentbound_jsde}. In particular,
Assumption~\ref{ass:ass}~\eqref{it:1bnd}--\eqref{it:bnd} are also
found in \cite[Condition 1]{ergodic_jsde}. In fact, these very
conditions can be shown to be exactly what is needed to apply the
earlier and quite general theory found in
\cite[Theorem~7.1]{FosterLyapunov}.

In Assumption~\ref{ass:ass}~\eqref{it:bnd} the case $\beta_{2} = 0$
will merit special attention. For well-posedness it turns out that we
will need to require a higher regularity of the initial data when
$\beta_{2} > 0$ (see Theorem~\ref{th:exist}) and the condition for
ergodicity becomes more restrictive (see
Theorem~\ref{th:ergodicity}). In practice, $\beta_{2} = 0$ implies
that opposing quadratic reactions of the type
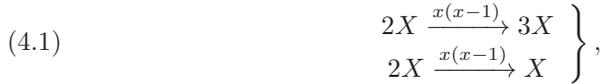
\begin{align}
   \left. \begin{array}{c}
     2X \xrightarrow{x(x-1)} 3X \\
     2X \xrightarrow{x(x-1)} X
 \end{array} \right\},
\end{align}
are impossible. Similarly, when $\beta_{1} = 0$ reactions of the type
\begin{align}
  X \xrightarrow{x} 2X
\end{align}
are excluded.

Note that \eqref{it:bnd} and \eqref{it:1lip} are redundant in the
sense that they are both implied by \eqref{it:lip}. However, as we saw
in Section~\ref{subsec:constants}, in \eqref{it:1lip} it is often
possible to find sharper constants $M$ and $\mu$ by considering this
bound in isolation. Also, although \eqref{it:lip} is stronger than
\eqref{it:1lip}, it is in particular valid for quadratic propensities
as can be seen from the representation used in the proof of
Proposition~\ref{prop:quadratic},
\begin{align}
  |w_{r}(x)-w_{r}(y)| &= |(x+y)^{T}S(x-y)| \le 
  \|S\|\|x+y\|_{1}\|x-y\|.
\end{align}

\subsubsection{The danger with cubic propensities}

Assumption~\ref{ass:ass}~\eqref{it:bnd} specifies the discussion to
propensities with at most quadratic growth, at least when measured in
the direction of the weight vector $\lvec$. To show that this is
natural we now demonstrate that additional care should be taken when
considering \emph{cubic} propensities.

\begin{example}
  \label{ex:cubic}
  Consider the model
  \begin{align*}
    \left. \begin{array}{c}
      3X \xrightarrow{x(x-1)(x-2)/2} X \\
      3X \xrightarrow{x(x-1)(x-2)} 4X 
  \end{array} \right\},
  \end{align*}
  such that the stoichiometric vector is given by $\stoich = [-2,1]$,
  and hence that the drift $-\stoich w(x) = 0$.
\end{example}

\begin{proposition}
  For the model in Example~\ref{ex:cubic}, if $X_{0} \ge 3$, then the
  second moment explodes in finite time.
\end{proposition}


\begin{proof}
  Assume that both the second and the third moment are bounded for $t
  \in [0,T)$ with $T > 0$. From \eqref{eq:dDynkin} we get the
  governing equation
  \begin{align*}
    \frac{d}{dt} EX_{t}^{2} &= E \, [3X_{t}(X_{t}-1)(X_{t}-2)],
  \end{align*}
  such that the growth of the second moment remains bounded only
  provided that the third moment remains finite. It is convenient to
  look at the cumulative third order moment. From \eqref{eq:dDynkin},
  \begin{align*}
    \frac{d}{dt} E \, C_{3}(X_{t}) := 
    \frac{d}{dt} E \, X_{t}(X_{t}-1)(X_{t}-2) &= E \, [9 C_{3}(X_{t}) (X_{t}-2/3)].
  \end{align*}
  By the arithmetic-geometric mean inequality, $x-2/3 \ge x-1 \ge
  [x(x-1)(x-2)]^{1/3}$, such that by Jensen's inequality,
  \begin{align*}
    \frac{d}{dt} E \, C_{3}(X_{t}) &\ge 9 E \left[
      C_{3}(X_{t})^{4/3} \right] \ge 9 [E \, C_{3}(X_{t})]^{4/3}.
  \end{align*}
  We put $u^{3} = E \, C_{3}(X_{t})$ and get the differential
  inequality
  \begin{align*}
    \frac{d(1/u)}{dt} &\le -3,
  \end{align*}
  which can be integrated and rearranged to produce the bound
  \begin{align}
    E \, [X_{t}(X_{t}-1)(X_{t}-2)] \ge
    \frac{X_{0}(X_{0}-1)(X_{0}-2)}{1-3t \, X_{0}(X_{0}-1)(X_{0}-2)}.
 \end{align}
 Hence the third, and consequently also the second moment explode for
 some finite $t$ whenever $X_{0} \ge 3$.
\end{proof}

Interestingly, we note that if $X_{0} = 3$, then the probability that
the cubic decay transition occurs first is $1/3$, and if this happens
the state of the system will be stuck with a single molecule
indefinitely.

\subsection{Moment bounds}
\label{subsec:moments}

In this section we consider general moment bounds derived from
\eqref{eq:Dynkin} using localization. To get some guidance, let us
first assume that the differential form of Dynkin's formula
\eqref{eq:dDynkin} is valid. Since any trajectory $(X_{t})_{t \ge 0}$
by the basic Assumption~\ref{ass:prop0} will belong to
$\Intdom_{+}^{D}$, we may use that $\lnorm{X_{t}} =
(\lvec,X_{t})$. Hence from \eqref{eq:dDynkin} with $f(x) =
(\lvec,x)$ we get that
\begin{align}
  \label{eq:E1bound0}
  \frac{d}{dt} E\lnorm{X_{t}} &= (\lvec,F(X_{t}))  \le A+\alpha \lnorm{X_{t}},
\end{align}
by Assumption~\ref{ass:ass}~\eqref{it:1bnd}. Clearly, the differential
form of Grönwall's inequality in Lemma~\ref{lem:gronwall} applies
here. A correct version of this argument unfortunately looses the sign
of $\alpha$.

\begin{proposition}
  \label{prop:E1bound}
  If Assumption~\ref{ass:ass}~\eqref{it:1bnd} is true, then
  \begin{align*}
    E \lnorm{X_{t}} &\le \lnorm{X_{0}} \exp (\alpha_{+} t)+
    A (\exp (\alpha_{+} t)-1)/\alpha_{+},
  \end{align*}
  where $\alpha_{+} = \alpha \vee 0$.
\end{proposition}

Here and below we shall make use of the stopping time
$\tau_{\stopping} = \inf_{t \ge 0}\{\lnorm{X_{t}} > \stopping\}$ and
define $\hat{t} = t \wedge \tau_{\stopping}$.

\begin{proof}
  From \eqref{eq:Dynkin} with $f(x) = (\lvec,x)$ we get that
  \begin{align}
    E\lnorm{X_{\hat{t}}} &= \lnorm{X_{0}}+E \int_{0}^{\hat{t}}
    (\lvec,F(X_{s})) \, ds \le
    \lnorm{X_{0}}+E \int_{0}^{t}
    A+\alpha_{+} \lnorm{X_{\hat{s}}} \, ds,
  \end{align}
  By the integral form of Grönwall's inequality in
  Lemma~\ref{lem:gronwall} we deduce in terms of $Y_{t} := X_{t \wedge
    \tau_{\stopping}}$ that
  \begin{align}
    E \lnorm{Y_{t}} &\le \lnorm{X_{0}} \exp (\alpha_{+} t)+
    A (\exp (\alpha_{+} t)-1)/\alpha_{+}
  \end{align}
  such that the same bound holds for $X_{t}$ by letting $\stopping \to
  \infty$.
\end{proof}

We attempt a similar treatment for obtaining bounds in mean
square. Assuming tactically that \eqref{eq:dDynkin} is valid, writing
$\|x\|^{2} = x^{T}x$ we get after some work that
\begin{align}
  \label{eq:E2norm}
  \frac{d}{dt} E\|X_{t}\|^{2} &= 
  E \left[ \onevect \stoich^{2} w(X_{t})-
    2X_{t}^{T}\stoich w(X_{t}) \right] \\
  \intertext{where $\stoich^{2}_{ij} \equiv (\stoich_{ij})^{2}$. We
    expect from Grönwall's inequality that $E \|X_{t}\|^{2}$ grows at
    most exponentially with $\alpha t$ whenever}
  \label{eq:tentativecond}
  \onevect \stoich^{2} w(x)&-2x^{T}\stoich w(x) \le 
  A+\alpha \|x\|^{2}.
\end{align}
However, this tentative condition is often violated in practice since
the second term $-x^{T}\stoich w(x)$ $= (x,F(x))$, and since we
already know from Proposition~\ref{prop:reversible1} that this
quantity does not admit bounds in terms of $\|x\|^{2}$ even for very
simple problems.

More realistic conditions arise when seeking to bound
$\lnorm{X_{t}}^{2}$ instead.

\begin{proposition}
  \label{prop:E2bound}
  If for some constants $\gamma$ and $C$,
  \begin{align}
    \label{eq:E2cond}
    (\lvect \stoich)^{2} w(x)-2\lvect \stoich w(x)\lnorm{x} \le
    C+\gamma \lnorm{x}^{2},
  \end{align}
  ($(\lvect \stoich)^{2}$ understood elementwise), then $E
  \lnorm{X_{t}}^{2} \le \lnorm{X_{0}}^{2} \exp (\gamma_{+}t)+C(\exp
  (\gamma_{+} t)-1)/\gamma_{+}$.
\end{proposition}

The proof of Proposition~\ref{prop:E2bound} follows the same pattern
as for Proposition~\ref{prop:E1bound}, but using this time $f(x) =
\lnorm{x}^{2} = (\lvec,x)^{2}$ in \eqref{eq:Dynkin}. The condition
\eqref{eq:E2cond} is typically more realistic than
\eqref{eq:tentativecond} since we recognize the term $-\lvect \stoich
w(x)\lnorm{x} = (\lvec,F(x))\lnorm{x}$, which under the evidently
reasonable Assumption~\ref{ass:ass}~\eqref{it:1bnd} is $\le
(A+\alpha\lnorm{x})\lnorm{x}$. It follows that if $(\lvect
\stoich)^{2} w(x)$ grows at most quadratically with $\lnorm{x}$, then
this assumption is sufficient to yield bounds in mean square. Stated
formally,
\begin{proposition}
  \label{prop:E2bound2}
  Under Assumption~\ref{ass:ass}~\eqref{it:1bnd} and \eqref{it:bnd}
  the condition \eqref{eq:E2cond} of Proposition~\ref{prop:E2bound} is
  true with $\gamma = 2\beta_{2}+2\alpha+2$ and $C =
  2B+\beta_{1}^{2}+A^{2}$.
\end{proposition}

\begin{proof}
  This is straightforward: we get by the assumptions and Hölder's
  inequality,
  \begin{align*}
    (\lvect \stoich)^{2} w(x)-2\lvect \stoich w(x)\lnorm{x} &\le
    2B+2\beta_{1} \lnorm{x}+2\beta_{2}\lnorm{x}^{2}+
    2(A+\alpha\lnorm{x})\lnorm{x},
  \end{align*}
  where an application of Young's inequality yields the indicated
  bounds.
\end{proof}

As a strong point in favor of our running assumptions we now
demonstrate that the above reasoning can be generalized: these two
conditions implies finite time stability in any order moment. We note
that in a recent manuscript \cite{momentbound_jsde}, related
conditions for the same results are proposed.

\begin{theorem}[\textit{Moment estimate}]
  \label{th:Epbound}
  Under Assumption~\ref{ass:ass}~\eqref{it:1bnd} and \eqref{it:bnd},
  for any integer $p \ge 1$,
  \begin{align}
    E\lnorm{X_{t}}^{p} &\le (\lnorm{X_{0}}^{p}+1) \exp (Ct)-1,
  \end{align}
  where $C > 0$ is a constant depending on the assumptions and on $p$. 
\end{theorem}

The proof of Theorem~\ref{th:Epbound} and some later results will
simplify using the following bound.

\begin{lemma}
  \label{lem:diffbound}
  Let $H(x) \equiv (x+y)^{p}-x^{p}$ with $x \in \Realdom_{+}$ and $y
  \in \Realdom$. Then for integer $p \ge 1$ we have the bounds
  \begin{align}
    \label{eq:signed_diffbound}
    H(x) &\le p y x^{p-1}+
    2^{p-4} p(p-1) y^{2} \left[ x^{p-2}+
      |y|^{p-2} \right], \\
   \label{eq:abs_diffbound}
    |H(x)| &\le p |y| 2^{p-2} \left[ x^{p-1}+
      |y|^{p-1} \right].
  \end{align}
\end{lemma}

\begin{proof}
  Both results follow from Taylor expansions;
  \begin{align*}
    H(x) &= p y x^{p-1}+\frac{p(p-1)}{2} y^{2} 
    \left[ x+\theta_{1} y \right]^{p-2}, \\
   |H(x)| &= p |y| \left| x+\theta_{2} y \right|^{p-1},
  \end{align*}
  respectively, where $\theta_{1,2} \in [0,1]$. Using the triangle
  inequality and the elementary inequality $(a+b)^{p} \le
  2^{p-1}(a^{p}+b^{p})$ the lemma is proved.
\end{proof}

\begin{proof}[\textit{Proof of Theorem~\ref{th:Epbound}}]
  Using \eqref{eq:Dynkin} with $f(X_{t}) \equiv [\lvect X_{t}]^{p}$
  we get
  \begin{align}
    \label{eq:Gdef}
    E\lnorm{X_{\hat{t}}}^{p} &= \lnorm{X_{0}}^{p}+
    E\int_{0}^{\hat{t}} \underbrace{\sum_{r = 1}^{R} w_{r}(X_{s}) \left[ 
      \left[ \lvect (X_{s}-\stoich_{r}) \right]^{p}-
      \left[ \lvect X_{s} \right]^{p} \right]}_{=: G(X_{s})} \, ds.
  \end{align}
  Using Lemma~\ref{lem:diffbound} \eqref{eq:signed_diffbound} and
  Assumption~\ref{ass:ass}~\eqref{it:1bnd} and \eqref{it:bnd} we
  obtain
  \begin{align*}
    G(x) &\le p(A+\alpha\lnorm{x}) \lnorm{x}^{p-1}+ \\
    &\phantom{\le}
    2^{p-3} p(p-1) (B+\beta_{1}\lnorm{x}+\beta_{2}\lnorm{x}^{2})
    (\lnorm{x}^{p-2}+\Delta^{p-2}),
  \end{align*}
  where $\Delta := \|\lvect \stoich\|_{\infty}$. Expanding and using
  Young's inequality with exponents $\{p/(p-1),p/(p-2)\}$ and
  conjugate exponents $\{p,p/2\}$, we get a bound
  \begin{align*}
    E\lnorm{X_{\hat{t}}}^{p} &\le \lnorm{X_{0}}^{p}+
    \int_{0}^{t} C(1+E\lnorm{X_{\hat{s}}}^{p}) \, ds,
  \end{align*}
  for some constant $C$ which thus depends on the
  assumptions. Applying Grönwall's inequality and letting $\stopping
  \to \infty$ we obtain the stated result.
\end{proof}

\subsection{Existence and uniqueness}
\label{subsec:wellposed}

We shall now prove that the jump SDE \eqref{eq:SDE} under
Assumption~\ref{ass:ass} has a uniquely defined and locally bounded
solution. To this end and following \cite[Sect.~3.1.2]{jumpSDEs}, we
introduce the following spaces of path-wise locally bounded processes:
\begin{align}
  \Sspace{p} &= \left\{ \begin{array}{l}
    X(t,\Probelem): \; X \mbox{ is
    $\Probfiltr_{t}$-adapted and $\Intdom_{+}^{D}$-valued such that } \\
  E \, \sup_{t \in [0,T]} \|X(t,\Probelem)\|_{1}^{p} <
  \infty \mbox{ for } \forall T < \infty \end{array} \right\}.
\end{align}

\begin{theorem}[\textit{Existence}]
  \label{th:exist}
  Let $X_{t}$ be a solution to \eqref{eq:SDE} under
  Assumption~\ref{ass:ass}~\eqref{it:1bnd} and \eqref{it:bnd} with
  $\beta_{2} = 0$. Then if $\lnorm{X_{0}}^{p} < \infty$, $\{X_{t}\}_{t
    \ge 0} \in \Sspace{p}$. If $\beta_{2} > 0$ then the conclusion
  remains under the additional requirement that $\lnorm{X_{0}}^{p+1} <
  \infty$.
\end{theorem}

\begin{proof}
  Below we let $C$ denote a positive constant which may be different
  on each occasion used. As before we use the stopping time
  $\tau_{\stopping} = \inf_{t \ge 0}\{\lnorm{X_{t}} > \stopping\}$ and
  put $\hat{t} = t \wedge \tau_{\stopping}$. We get from It\^o's
  formula (with $G$ defined in \eqref{eq:Gdef})
  \begin{align*}
    \lnorm{X_{\hat{t}}}^{p} &= \lnorm{X_{0}}^{p}+\int_{0}^{\hat{t}} 
    G(X_{s}) \, ds+ \\
    &\phantom{=} \int_{0}^{\hat{t}} \int_{\markspace}
    \sum_{r = 1}^{R} \hatw_{r}(X_{s-}; \; z) \left[ 
      \left[ \lvect (X_{s-}-\stoich_{r}) \right]^{p}-
      \left[ \lvect X_{s-} \right]^{p} \right] \, (\mu-m)(ds \times dz).
  \end{align*}
  Since the propensities are bounded for bounded arguments
  (Assumption~\ref{ass:prop0}), using the stopping time we find that
  the jump part is absolutely integrable and hence a local martingale
  $M_{\hat{t}}$. We estimate its quadratic variation under
  Assumption~\ref{ass:ass}~\eqref{it:bnd},
  \begin{align}
    \nonumber
    E \, [M]_{\hat{t}}^{1/2} &\le
    E \left[ \int_{0}^{\hat{t}} \int_{\markspace}
      \sum_{r = 1}^{R} \hatw_{r}(X_{s}; \; z) \left| 
      \left[ \lvect (X_{s}-\stoich_{r}) \right]^{p}-
      \left[ \lvect X_{s} \right]^{p} \right| \, \mu(ds \times dz) \right] \\
    \nonumber
    &= E \left[ \int_{0}^{\hat{t}} \int_{\markspace}
      \sum_{r = 1}^{R} \hatw_{r}(X_{s}; \; z)W(X_{s}) \left| 
      \left[ \lvect (X_{s}-\stoich_{r}) \right]^{p}-
      \left[ \lvect X_{s} \right]^{p} \right| \, ds \times dz \right] \\
    \nonumber
    &= E \left[ \int_{0}^{\hat{t}} 
      \sum_{r = 1}^{R} w_{r}(X_{s}) \left| 
      \left[ \lvect (X_{s}-\stoich_{r}) \right]^{p}-
      \left[ \lvect X_{s} \right]^{p} \right| \, ds \right] \\
    \label{eq:M1}
    &\le E \left[ \int_{0}^{\hat{t}} 
      \sum_{r = 1}^{R} p |\lvect \stoich_{r}| w_{r}(X_{s}) \; 2^{p-2}
      \left[ \lnorm{X_{s}}^{p-1}+|\lvect \stoich_{r}|^{p-1} \right]
      \, ds \right] \\
    \label{eq:M2}
    &\le E \left[ \int_{0}^{\hat{t}} 
      C(B+\beta_{1}\lnorm{X_{s}}+\beta_{2}\lnorm{X_{s}}^{2})
      \left[ \lnorm{X_{s}}^{p-1}+\Delta^{p-1} \right]
      \, ds \right] \\
    \label{eq:M3}
     &\le E \int_{0}^{\hat{t}} C(1+\lnorm{X_{s}}^{p}+
    \beta_{2}\lnorm{X_{s}}^{p+1}) \, ds
  \end{align}
  where $\Delta = \|\lvect \stoich\|_{\infty}$. In \eqref{eq:M1}
  Lemma~\ref{lem:diffbound}~\eqref{eq:abs_diffbound} was applied and
  Assumption~\ref{ass:ass}~\eqref{it:bnd} entered in
  \eqref{eq:M2}. Assume first that $\beta_{2} = 0$. Then for the drift
  part we have already constructed a suitable bound in
  Theorem~\ref{th:Epbound} such that
  \begin{align*}
    \lnorm{X_{\hat{t}}}^{p} &\le \lnorm{X_{0}}^{p}+\int_{0}^{\hat{t}} 
    C(1+\lnorm{X_{s}}^{p}) \, ds+|M_{\hat{t}}|.
  \end{align*}
  Taking supremum and expectation values we get from Burkholder's
  inequality \cite[Chap.~IV.4]{protterSDE} that
  \begin{align*}
    E \sup_{s \in [0,\hat{t}]} \lnorm{X_{s}}^{p} &\le
    \lnorm{X_{0}}^{p}+\int_{0}^{\hat{t}}
    C(1+E\sup_{s' \in [0,s]} \lnorm{X_{s'}}^{p}) \, ds.
  \end{align*}
  Writing $\lnorm{X}^{p}(t) \equiv \sup_{s \in [0,t]}
  \lnorm{X_{s}}^{p}$ we conclude that
  \begin{align*}
    E \lnorm{X}^{p}(t \wedge \tau_{\stopping}) &\le \lnorm{X_{0}}^{p}+\int_{0}^{t}
    C(1+E \lnorm{X}^{p}(s \wedge \tau_{\stopping})) \, ds.
  \end{align*}
  By Grönwall's inequality we have thus shown that $E \lnorm{X}^{p}(t
  \wedge \tau_{\stopping})$ can be bounded in terms of the initial
  data and time $t$. The result now follows by letting $\stopping \to
  \infty$ and using Fatou's lemma.

  Next assume that $\beta_{2} > 0$. Then we have to rely more directly
  on Theorem~\ref{th:Epbound} in \eqref{eq:M3},
  \begin{align*}
    E \, [M]_{\hat{t}}^{1/2} &\le
    \int_{0}^{\hat{t}} 
    C(1+E\lnorm{X_{s}}^{p+1}) \, ds 
    \le (e^{C\hat{t}}-1)(\lnorm{X_{0}}^{p+1}+1),
  \end{align*}
  where, although there is now a dependency on $\|X_{0}\|^{p+1}$, the
  rest of the argument carries through.
\end{proof}

For the case that the initial data $X_{0}$ is non-deterministic we see
that the general quadratic case
Assumption~\ref{ass:ass}~\eqref{it:bnd} with $\beta_{2} > 0$ requires
a one order higher moment of the initial data in order for a solution
in $\Sspace{p}$ to exist.

\begin{theorem}[\textit{Uniqueness}]
  \label{th:unique}
  Let Assumption~\ref{ass:ass}~\eqref{it:1bnd}--\eqref{it:1lip} hold
  true. Then any two paths $X_{t}$ and $Y_{t}$ coupled according to
  the description in Section~\ref{subsubsec:coupling} with $X_{0} =
  Y_{0}$ are equal.
\end{theorem}

We shall be using the observation that, for $x \in \Intdom_{+}^{D}$,
we have that $\|x\|_{1} \le \|x\|^{2}$ (referred below to as the
``integer inequality'').

\begin{proof}
  Under the same stopping time as before we get from It\^o's formula
  using the coupling described in Section~\ref{subsubsec:coupling}
  that
  \begin{align}
    \nonumber
    E\|X_{\hat{t}}-Y_{\hat{t}}\|^{2} &= E\,\int_{0}^{\hat{t}}
    2(X_{s}-Y_{s},F(X_{s})-F(Y_{s})) \\
    \nonumber
    &\hphantom{= E\,\int_{0}^{\hat{t}}}
    +(\onevect \stoich^{2})|w(X_{s})-w(Y_{s})|\, ds \\
    \nonumber
    &\le E\,\int_{0}^{\hat{t}}
    2(M+\mu\|X_{s}+Y_{s}\|_{1})\|X_{s}-Y_{s}\|^{2} \\
    \label{eq:E2flowbnd}
    &\hphantom{= E\,\int_{0}^{\hat{t}}}
    +\|\onevect \stoich^{2}\|_{\infty} 
    L(1+\|X_{s}+Y_{s}\|_{1})\|X_{s}-Y_{s}\| \, ds.
    \intertext{From the integer inequality we find that there is a
      constant $C \ge 0$ depending on $\stopping$ such that}
    \nonumber
    E\|(X-Y)(t \wedge \tau_{\stopping})\|^{2} &\le 
    \int_{0}^{t \wedge \tau_{\stopping}} C E\|X_{s}-Y_{s}\|^{2} \, ds
    \le \int_{0}^{t} C E\|(X-Y)(s \wedge \tau_{\stopping})\|^{2} \, ds.
  \end{align}
  Using that $E\|X_{0}-Y_{0}\|^{2} = 0$ and Grönwall's inequality we
  conclude that the only solution is the zero solution. Letting
  $\stopping \to \infty$ and using Fatou's lemma the statement is
  therefore proved.
\end{proof}

In a certain sense the previous result is trivial; from the Poisson
representation \eqref{eq:Poissrepr} we see that up to the first
explosion, a path is uniquely determined from an initial state and a
series of Poisson distributed events. However, and as we shall see
below, the above proof is prototypical for more involved
situations. An example would be when devising hybrid approximations in
continuous state space. Indeed, in the above proof, note that if the
integer inequality did not hold we would naturally have to rely on the
Cauchy-Schwartz inequality instead. With $u(t)^{2} =
E\|X_{t}-Y_{t}\|^{2}$ this leads to bounds of the typical kind
\begin{align}
  u(t)^{2} &\le \int_{0}^{t} C(u(s)+u(s)^{2}) \, ds,
\end{align}
for which $u(t) = \exp(Ct/2)-1$ is an admissible solution. This
observation shows that the integer inequality as used in the proof is
crucial; without it the integral inequality \eqref{eq:E2flowbnd}
admits growing solutions.

\subsection{Stability}
\label{subsec:stability}

Although Theorem~\ref{th:Epbound} shows that any moments are bounded
in finite time, a relevant question from the modeling point of view is
whether the first few moments remain bounded indefinitely. We give a
result to this effect which relies on the existence of solutions in
$\Sspace{p}$ which implies that the differential form
\eqref{eq:dDynkin} of Dynkin's formula may be used
(cf.~Corollary~\ref{cor:mastervalid}) such that in turn the
differential Grönwall inequality applies. We mention anew that a very
similar result has recently appeared in \cite[Theorem
  2]{ergodic_jsde}.

\begin{theorem}[\textit{Ergodicity}]
  \label{th:ergodicity}
  Under Assumption~\ref{ass:ass}~\eqref{it:1bnd}--\eqref{it:bnd},
  suppose that 
  \begin{align}
    \alpha+\beta_{2}(p-1) \le
    -\kappa_{p} < 0.
  \end{align}
  Then for integer $p \ge 1$, $E\lnorm{X_{t}}^{p}$ remains bounded as
  $t \to \infty$.
\end{theorem}

\begin{proof}
  The case $\beta_{2} > 0$ is slightly more complicated to obtain so
  we shall concentrate on this. We omit the case $p = 1$ since it
  follows from \eqref{eq:E1bound0} under the present assumptions. The
  idea of the proof is to asymptotically bound
  $E(C+\lnorm{X_{t}})^{p}$ with a certain positive constant $C = C(p)$
  to be decided upon below. By \eqref{eq:dDynkin} we get with $Z_{t} =
  C+\lnorm{X_{t}}$,
  \begin{align}
    \nonumber
    \frac{dE Z_{t}^{p}}{dt} 
    &= E\sum_{r = 1}^{R} w_{r}(X_{t}) \left[ 
      \left[ Z_{t}-\lvect \stoich_{r} \right]^{p}-
      Z_{t}^{p} \right] \\
    \label{eq:dEZ}
    = \; &E\sum_{r = 1}^{R} w_{r}(X_{t}) 
    \Bigl[ -\lvect \stoich_{r} pZ_{t}^{p-1}
      +\frac{(-\lvect \stoich_{r})^{2}}{2} p(p-1)
      (Z_{t}-\theta_{r}\lvect \stoich_{r})^{p-2} \Bigr],
  \end{align}
  by Taylor's formula for some $\theta_{r} \in [0,1]$. Using the
  assumptions we get the bound
  \begin{align}
    \nonumber
    \frac{dE Z_{t}^{p}}{dt} 
    &\le p E\, \Bigl[ (A+\alpha \lnorm{X_{t}})Z_{t}^{p-1} \\
      \label{eq:bound}
      &\phantom{\le p E\,}+
      (p-1)(B+\beta_{1}\lnorm{X_{t}}+\beta_{2}\lnorm{X_{t}}^{2})(Z_{t}+
      \|\lvect \stoich\|_{\infty})^{p-2} \Bigr].
  \end{align}
  For the first term in \eqref{eq:bound} we get from the scaled
  Young's inequality with exponent $p/(p-1)$ and conjugate exponent
  $p$ that
  \begin{align*}
    (A+\alpha \lnorm{X_{t}})Z_{t}^{p-1} &= 
    \alpha Z_{t}^{p}+(A-\alpha C)Z_{t}^{p-1} \\
    &\le \alpha Z_{t}^{p}+\frac{\epsilon (p-1)}{p} Z_{t}^{p}+
    \frac{\epsilon^{1-p}}{p}|A-\alpha C|^{p},
  \end{align*}
  for some $\epsilon > 0$. As for the second term in \eqref{eq:bound}
  we first estimate for $\beta_{2} > 0$
  \begin{align*}
    B+\beta_{1}\lnorm{X_{t}}+\beta_{2}\lnorm{X_{t}}^{2} \le
    \beta_{2} \left( \lnorm{X_{t}}+B/\beta_{1} \right)
    \left( \lnorm{X_{t}}+\beta_{1}/\beta_{2} \right).
  \end{align*}
  Next by the arithmetic-geometric mean inequality we get
  \begin{align*}
    \left( \lnorm{X_{t}}+B/\beta_{1} \right)&
    \left( \lnorm{X_{t}}+\beta_{1}/\beta_{2} \right)
    (Z_{t}+\|\lvect \stoich\|_{\infty})^{p-2} \\
    \le &\left( \frac{B}{p\beta_{1}}+
    \frac{\beta_{1}}{p\beta_{2}}+
    \frac{p-2}{p}(\|\lvect \stoich\|_{\infty}+C)+\lnorm{X_{t}} \right)^{p}
    = Z_{t}^{p},
  \end{align*}
  provided that we choose $C$ as the solution to the equation
  \begin{align}
    \label{eq:Cdef}
    \frac{B}{p\beta_{1}}+
    \frac{\beta_{1}}{p\beta_{2}}+
    \frac{p-2}{p}(\|\lvect \stoich\|_{\infty}+C) &= C.
  \end{align}
  Taken together we thus have
  \begin{align*}
    \frac{dE Z_{t}^{p}}{dt} &\le 
    p \left( -\kappa_{p}+
    \frac{\epsilon (p-1)}{p} \right) E Z_{t}^{p}+
    \epsilon^{1-p}|A-\alpha C|^{p}.
  \end{align*}
  Since $\kappa_{p} > 0$ we may pick a small enough $\epsilon$ such
  that the bracketed expression remains negative. By Grönwall's
  inequality this then proves the result with $\beta_{2} > 0$. To
  prove the case $\beta_{2} = 0$ the same idea of proof applies and
  results in
  \begin{align*}
    \frac{dE Z_{t}^{p}}{dt} &\le 
    p \left( -\kappa_{p}+
    \frac{\epsilon (p-1)}{p}+
    \frac{\epsilon(p-1)^{2}}{p}\beta_{1}
    \right) E Z_{t}^{p} \\
    &\phantom{\le \frac{p}{2}}+
    \epsilon^{1-p}\left(|A-\alpha C|^{p}+
      \beta_{1}(p-1) \right),
  \end{align*}
  for a certain new constant $C$ satisfying an equation similar to
  \eqref{eq:Cdef}.
\end{proof}


We next aim at deriving some stability estimates with respect to
perturbations in the \emph{reaction coefficients}. An early account of
this was given by Kurtz in \cite{countingProcesses}, see also
\cite{parameterCTMC} for a recent discussion in a bounded
setting. Given the linear dependence on the coefficients $k_{r}$, $r =
1 \ldots 4$ in the elementary reactions \eqref{eq:el} a suitable model
seems to be that a perturbation $k_{r} \to k_{r}+\delta_{r}$ in a
propensity $w_{r}(x)$ spreads linearly in a relative sense,
\begin{align}
  |w_{r}(x,k_{r})-w_{r}(x,k_{r}+\delta_{r})| &\le \mbox{constant} \times 
  \delta_{r} w_{r}(x,k_{r}). \\
  \intertext{We make this formal by requiring that}
  \label{eq:coeffLip}
  (\onevect \stoich^{2})|w(x)-w_{\delta}(x)| &\le \delta \|w(x)\|_{1}
  \le \delta C(1+\|x\|_{1}^{2}),
  \intertext{where $\delta$ is a suitable measure of the total
    perturbation vector and where the perturbed propensity vector
    function is given by}
  \label{eq:perturbation}
  w_{\delta}(x) \equiv [w_{1}(x,k_{1}+\delta_{1}),&
    \ldots,w_{R}(x,k_{R}+\delta_{R})]^{T}.
\end{align}
The existence of an absolute constant $C$ in \eqref{eq:coeffLip}
follows from Assumption~\ref{ass:ass}~\eqref{it:lip}. We further
conveniently assume that the entire statement of
Assumption~\ref{ass:ass} carries over to the perturbed system, and for
convenience we shall also assume that all constants are the same. By
the triangle inequality and Assumption~\ref{ass:ass}~\eqref{it:lip} we
obtain from \eqref{eq:coeffLip} the bound
\begin{align}
  \nonumber
  (\onevect \stoich^{2})|w(x)-w_{\delta}(y)| &\le
  L(1+\|x+y\|_{1})\|x-y\|+\delta C(1+\|y\|_{1}^{2}) \\
  \label{eq:coeffLip2} 
  &\le C(\delta+\|x-y\|^{2}),
\end{align}
with $C$ some constant and where the simplification in
\eqref{eq:coeffLip2} assumes an \textit{a priori} bound (e.g.~stopping
time) $\|x+y\|_{1} \le P$ and additionally requires the integer
inequality.

The starting point for the analysis will be It\^o's formula under the
coupling described in Section~\ref{subsubsec:coupling}. The techniques
used below generalize well to $p$th order moment estimates, but for
ease of exposition we let $p = 2$.

Hence under the model for coefficient perturbations
\eqref{eq:coeffLip}--\eqref{eq:perturbation} we have that
\begin{align}
    \nonumber
    \|X_{t}-Y_{t}\|^{2} &= \int_{0}^{t}
    2(X_{s}-Y_{s},F(X_{s})-F_{\delta}(Y_{s}))
    +(\onevect \stoich^{2})|w(X_{s})-w_{\delta}(Y_{s})| \, ds \\
    \nonumber
    &\phantom{=}+\int_{0}^{t} \int_{\markspace}
    2(X_{s-}-Y_{s-},-\stoich(\hatw^{\delta}(X_{s-},Y_{s-}; \,
    z)-\hatw^{\delta}(Y_{s-},X_{s-}; \, z)))+ \\
    \label{eq:2flow}
    &\phantom{+\int_{0}^{t}}
    \|-\stoich(\hatw^{\delta}(X_{s-},Y_{s-}; \, z)-
    \hatw^{\delta}(Y_{s-},X_{s-}; \, z)))\|^{2} \, (\mu_{\delta}-m_{\delta})(ds \times dz).
\end{align}

\begin{theorem}[\textit{Continuity}]
  \label{th:continuity}
  Let two trajectories $X_{t}$ and $Y_{t}$ be given, with the same
  initial data and coupled according to the discussion in
  Section~\ref{subsubsec:coupling}. Let the propensities for $Y_{t}$
  be perturbed by $\delta$ as indicated in
  \eqref{eq:coeffLip}--\eqref{eq:perturbation}. Then
  \begin{align}
    \lim_{\delta \to 0+} E \|X_{t}-Y_{t}\|^{2} = 0.
  \end{align}
\end{theorem}

\begin{proof}
  We use the stopping time $\tau_{\stopping} = \inf_{t \ge
    0}\{\|X_{t}+Y_{t}\|_{1} > \stopping\}$ and put $\hat{t} = t \wedge
  \tau_{\stopping}$. From \eqref{eq:2flow} we get
 \begin{align}
    \nonumber
    E\|X_{\hat{t}}-Y_{\hat{t}}\|^{2} = \int_{0}^{\hat{t}} E \bigr[&
    2(X_{s}-Y_{s},F(X_{s})-F_{\delta}(Y_{s}))
    +(\onevect \stoich^{2})|w(X_{s})-w_{\delta}(Y_{s})| \bigl] \, ds \\
    \label{eq:E2flow_delta}
    &\le \int_{0}^{\hat{t}} E \Bigl[
    2(M+\mu\|X_{s}+Y_{s}\|_{1})\|X_{s}-Y_{s}\|^{2}+ \\
    \nonumber
    2\delta(\onevect \stoich^{2})^{1/2} &\|w(Y_{s})-w_{\delta}(Y_{s})\|
    \|X_{s}-Y_{s}\|+
    C(\delta+\|X_{s}-Y_{s}\|^{2}) \Bigr] \, ds, \\
    \intertext{where \eqref{eq:coeffLip2} was used. Simplifying
      further for a bounded $\delta$ we get}
    \nonumber
    E\|(X-Y)(t \wedge \tau_{\stopping})\|^{2} &\le 
    \int_{0}^{t \wedge \tau_{\stopping}} C \left(
      \delta+E\|X_{s}-Y_{s}\|^{2} \right) \, ds \\
    \nonumber
    &\le \int_{0}^{t} C \left( \delta+E\|(X-Y)(s \wedge
      \tau_{\stopping})\|^{2} \right) \, ds.
  \end{align}
  Using that $X_{0} = Y_{0}$ and Grönwall's inequality we conclude
  that
  \begin{align*}
    E\|X_{\hat{t}}-Y_{\hat{t}}\|^{2} \le \delta (\exp(Ct)-1).
  \end{align*}
  To get rid of the stopping time we write in terms of indicator functions,
  \begin{align}
    \nonumber
    E\|X_{t}-Y_{t}\|^{2} &= E \left[ \|X_{t}-Y_{t}\|^{2}[t
      < \tau_{\stopping}]+\|X_{t}-Y_{t}\|^{2}[t \ge \tau_{\stopping}]
    \right] \\
    \label{eq:2est}
    &\le E \|X_{\hat{t}}-Y_{\hat{t}}\|^{2}+\left(
        E\|X_{t}-Y_{t}\|^{4} \right)^{1/2} \left( \Prob[t \ge
        \tau_{\stopping}] \right)^{1/2}. \\
    \intertext{Using $\Prob[t \ge \tau_{\stopping}] =
        \Prob[\sup_{s \in [0,t]} \|X_{s}+Y_{s}\|_{1} > \stopping]$
        we get from Markov's inequality and the previous estimate}
      \nonumber
      E\|X_{t}-Y_{t}\|^{2} &\le \delta  (\exp(Ct)-1) \\
      \label{eq:2est2}
      &+\left(
        E\|X_{t}-Y_{t}\|^{4} \right)^{1/2} P^{-1/2} \left( E \sup_{s
          \in [0,t]} \|X_{s}+Y_{s}\|_{1} \right) ^{1/2}.
  \end{align}
  Relying on the existence result in Theorem~\ref{th:exist} we find
  that for any given $\varepsilon > 0$ we can select $P$ (and hence
  also $C$) such that the right term is $< \varepsilon/2$. We can next
  find $\delta_{0} > 0$ such that for all $\delta \le \delta_{0}$,
  also the left term is $< \varepsilon/2$. Hence for all $\delta \le
  \delta_{0}$, $E\|X_{t}-Y_{t}\|^{2} < \varepsilon$ as claimed.
\end{proof}

As a by-product of the proof we see that if the process is
\emph{bounded}, then for $P$ large enough the probability in
\eqref{eq:2est} is zero.

\begin{corollary}[\textit{Perturbation estimate, bounded version}]
  If in Theorem~\ref{th:continuity}, the processes $X_{t}$ and $Y_{t}$
  are bounded, then for a constant $C > 0$,
  \begin{align}
    \label{eq:2estbnd}
    E \|X_{t}-Y_{t}\|^{2} \le \delta (\exp(Ct)-1).
 \end{align}
\end{corollary}

The constant $C$ in \eqref{eq:2estbnd} can be bounded explicitly by
inspection of \eqref{eq:coeffLip2} and \eqref{eq:E2flow_delta}.

For an \emph{unbounded} system it is apparently much more difficult to
obtain explicit estimates. However, by controlling also the martingale
part we can strengthen Theorem~\ref{th:continuity} in another
direction.

\begin{theorem}[\textit{Continuity/$\sup$}]
  \label{th:sup_continuity}
  Under the same assumptions as Theorem~\ref{th:continuity} we have that
  \begin{align}
    \lim_{\delta \to 0+} E \sup_{s \in [0,t]} \|X_{s}-Y_{s}\|^{2} = 0.
  \end{align}
\end{theorem}

\begin{proof}
  The quadratic variation of the martingale part in \eqref{eq:2flow}
  can be bounded as
  \begin{align*}
    \nonumber
    E \, [M]_{\hat{t}}^{1/2} &\le
    E \Biggl[ \int_{0}^{\hat{t}} \int_{\markspace}
      \Bigl| 2(X_{s-}-Y_{s-},-\stoich(\hatw^{\delta}(X_{s-},Y_{s-}; \,
    z)-\hatw^{\delta}(Y_{s-},X_{s-}; \, z)))+ \\
    &\phantom{\le \int_{0}^{t} \int_{\markspace}}
    \|\stoich(\hatw^{\delta}(X_{s-},Y_{s-}; \, z)-
    \hatw^{\delta}(Y_{s-},X_{s-}; \, z)))\|^{2}\Bigr| \, \mu_{\delta}(ds \times dz) \Biggr] \\
    &\le
    E \Biggl[ \int_{0}^{\hat{t}} \int_{\markspace}
      \Bigl( 2\|X_{s}-Y_{s}\|\|\stoich(\hatw^{\delta}(X_{s},Y_{s}; \,
    z)-\hatw^{\delta}(Y_{s},X_{s}; \, z))\|+ \\
    &\phantom{\le \int_{0}^{t} \int_{\markspace}}
    \|\stoich(\hatw^{\delta}(X_{s},Y_{s}; \, z)-
    \hatw^{\delta}(Y_{s},X_{s}; \, z)))\|^{2} \Bigr)
    W^{\delta}(X_{s},Y_{s}) \, ds
    \times dz \Biggr] \\
    &=
    E \Biggl[ \int_{0}^{\hat{t}} 
    \Bigl( 2\|X_{s}-Y_{s}\| (\onevect \stoich^{2})^{1/2}+
    (\onevect \stoich^{2}) \Bigr) |w(X_{s})-w_{\delta}(Y_{s})| \, ds \Biggr] \\
    &\le \int_{0}^{\hat{t}} C \left(
      \delta+E\|X_{s}-Y_{s}\|^{2} \right) \, ds,
 \end{align*}  
 after using \eqref{eq:coeffLip2} and the integer inequality anew. For
 the drift part we may use the corresponding bound developed in the
 proof of Theorem~\ref{th:continuity}. After taking supremum and
 expectation values of \eqref{eq:2flow} and using Burkholder's
 inequality we therefore arrive at
 \begin{align*}
   E \|X-Y\|^{2}(t \wedge \tau_{\stopping}) &\le \int_{0}^{\hat{t}} C (
      \delta+E \|X-Y\|^{2}(s) ) \, ds \\
      &\le \int_{0}^{t} C (
      \delta+E \|X-Y\|^{2}(s \wedge \tau_{\stopping})) \, ds \\
      &\le \delta (\exp(Ct)-1)
 \end{align*}
 by Grönwall's inequality and using the notation $\|X\|(t) \equiv
 \sup_{s \in [0,t]} \|X_{t}\|$. We now rely on the same strategy as in
 the proof of Theorem~\ref{th:continuity} to similarly arrive at
 \begin{align*}
   E\|X-Y\|^{2}(t) &\le \delta  (\exp(Ct)-1)+\left(
     E\|X-Y\|^{4}(t) \right)^{1/2} P^{-1/2} \left( E \|X+Y\|_{1}(t) \right) ^{1/2},
 \end{align*}
 and the conclusion follows as before.
\end{proof}


\section{Conclusions}
\label{sec:conclusions}

We have proposed a theoretical framework consisting of \textit{a
  priori} assumptions and estimates for problems in stochastic
chemical kinetics. The assumptions are strong enough to guarantee
well-posedness for a large and physically relevant class of
problems. Long time estimates and limit results for perturbations in
rate constants have been studied to exemplify the theory. The
assumptions are \emph{constructive} in the sense that explicit
techniques for obtaining all postulated constants have either been
worked out in detail or at least indicated. We have seen that the case
$\beta_{2} = 0$ in Assumption~\ref{ass:ass}~\eqref{it:bnd} is
particularly promising from the analysis point of view in that the
conditions for existence in Theorem~\ref{th:exist} and the ergodicity
in Theorem~\ref{th:ergodicity} both can be formulated naturally.

In the course of motivating our setup we have seen that most problems
do not admit global Lipschitz constants and that one-sided versions do
not provide a better alternative. Another conclusion worth
highlighting is that it pays off to consider jump SDEs in a fully
discrete setting in that there are potential complications in proving
uniqueness in continuous state space. A practical implication is that
care should be exercised when forming continuous approximations to
these types of jump SDEs.

For future work we intend to re-visit certain classical results from
the perspective of the framework developed herein; for example,
thermodynamic limit results, time discretization strategies, and
quasi-steady state approximations --- all of which have a practical
impact in a range of applications.


\section*{Acknowledgment}

The author likes to express his sincere gratitude to Takis
Konstantopoulos for several fruitful and clarifying discussions. Early
inputs on this work were also gratefully obtained from Henrik Hult,
Ingemar Kaj, and Per Lötstedt.

This work was supported by the Swedish Research Council within the
UPMARC Linnaeus center of Excellence.


\newcommand{\doi}[1]{\href{http://dx.doi.org/#1}{doi:#1}}
\newcommand{\available}[1]{Available at \url{#1}}
\newcommand{\availablet}[2]{Available at \href{#1}{#2}}


\providecommand{\noopsort}[1]{} \providecommand{\doi}[1]{\texttt{doi:#1}}
  \providecommand{\available}[1]{Available at \texttt{#1}}
  \providecommand{\availablet}[2]{Available at \texttt{#2}}

\end{document}